\documentclass[12pt]{article}
\usepackage[latin1]{inputenc}
\usepackage{a4,exscale,graphicx,amsmath,amssymb,hyperref,verbatim,capt-of}

\newtheorem{theorem}{Theorem}[section]

\newtheorem{corollary}{Corollary}[section]

\newtheorem{proposition}{Proposition}[section]

\newtheorem{definition}{Definition}[section]

\newenvironment{proof}{\noindent\textbf{Proof.}}{\hfill $\blacksquare$\\}

\newcommand{\D}{\displaystyle}

\newcommand{\s}{\scriptscriptstyle}

\newcommand{\IR}{\mathbb{R}}

\newcommand{\IH}{\mathbb{H}}
\newcommand{\sIH}{{\scriptscriptstyle\mathbb{H}}}

\newcommand{\IC}{\mathbb{C}}
\newcommand{\sIC}{{\scriptscriptstyle\mathbb{C}}}

\newcommand{\IN}{\mathbb{N}}
\newcommand{\INo}{\mathbb{N}_{0}}

\newcommand{\IZ}{\mathbb{Z}}

\newcommand{\sOm}{{\scriptscriptstyle\Omega}}

\newcommand{\IBp}{\mathbb{B}_{3}^{+}}
\newcommand{\sIBp}{{\scriptscriptstyle\mathbb{B}_{3}^{+}}}

\newcommand{\IBm}{\mathbb{B}_{3}^{-}}

\newcommand{\A}{\mathtt{A}}
\newcommand{\B}{\mathtt{B}}
\newcommand{\C}{\mathtt{C}}

\newcommand{\Ah}{\check{A}}

\DeclareMathOperator\sign{sign}

\title{On a three dimensional analogue to the holomorphic $z$-powers}
\author{S. Bock\thanks{Institute of Mathematics/Physics, Bauhaus-University, Weimar, Germany,
e-mail: sebastian.bock@uni-weimar.de}}
\date{}

\begin{document}
\maketitle \abstract{The main objective of this article is a constructive generalization of the
holomorphic power and Laurent series expansions in $\IC$ to dimension 3 using the framework of
hypercomplex function theory. For this reason, deals the first part of this article with
generalized Fourier \& Taylor series expansions in the space of square integrable quaternion-valued
functions which possess peculiar properties regarding the hypercomplex derivative and primitive. In
analogy to the complex one-dimensional case, both series expansions are orthogonal series with
respect to the unit ball in $\IR^3$ and their series coefficients can be explicitly (one-to-one)
linked with each other. Furthermore, very compact and efficient representation formulae
(recurrence, closed-form) for the elements of the orthogonal bases are presented. The latter
results are then used to construct a new orthonormal bases of outer solid spherical monogenics in
the space of square integrable quaternion-valued functions. This finally leads to the definition of
a generalized Laurent series expansion for the spherical shell.\\}


\textbf{This manuscript is a preprint version. An extended version with the complete proofs has
been submitted to Complex Variables \& Elliptic Equations \cite{Bock2010a,Bock2010b}.}\newpage

%%\tableofcontents
%%\newpage
%***********************************************************************
%%***********************************************************************
\section{Introduction}
%%***********************************************************************
%%***********************************************************************

Let $\mathbb{B}_{2}^{+}$ be the unit disc and $f$ a holomorphic function in the space of square
integrable holomorphic functions $L^{2}(\mathbb{B}_{2}^{+};\IC)\cap\ker\bar{\partial}$. As is well
known, an orthonormal basis of holomorphic polynomials $\left\{\tilde{z}^{n}\right\}_{n\in\INo}$ in
$L^{2}(\mathbb{B}_{2}^{+};\IC)\cap\ker\bar{\partial}$ is obtained by the normalization of the
complex monomials
\begin{equation*}
\tilde{z}^{n}\,=\,\sqrt{\frac{n+1}{\pi}}\,z^{n}
\end{equation*}
with respect to the unit disc. From the orthogonality and completeness of the above basis follows
directly that each function $f\in L^{2}(\mathbb{B}_{2}^{+};\IC)\cap\ker\bar{\partial}$ possesses a
unique Fourier series expansion
\begin{equation*}
f\,=\,\sum_{n=0}^{\infty} \tilde{z}^{n}\boldsymbol{\beta}_{n},\quad \text{with}\quad
\boldsymbol{\beta}_{n}\, = \, <\tilde{z}^{n},f>_{\s(\mathbb{B}_{2}^{+};\IC)}
\,=\,\int_{\mathbb{B}_{2}} \overline{\tilde{z}^{n}} f \, d\sigma.
\end{equation*}
A important structural property of the complex Fourier series is related to their complex
derivative and primitive. Therefore, let us firstly consider the non-normalized basis functions
$z^{n}$ with respect to their complex derivative $\partial_{z} z^{n} = n\,z^{n-1}$ as well as their
holomorphic primitive $\int z^{n} dz = \frac{1}{n+1}\,z^{n+1}$. Here, the significant property
results from the fact that the complex derivation or primitivation of an arbitrary basis function
yields again a real multiple of a single basis function. Formally speaking, this means that the
derivative as well as the primitive of the complex Fourier series can be considered again as a
Fourier series expansion. In addition, it is also well known that in the complex theory exists a
direct relation between the global and the local approximation of a function. Let $f$ be again a
holomorphic function in $L^{2}(\mathbb{B}_{2}^{+};\IC)\cap\ker\bar{\partial}$ then $f$ possesses in
$\mathbb{B}_{2}$ the Taylor series expansion
\begin{equation*}
f\,=\,\left.\sum_{n=0}^{\infty} z^{n}\,\frac{f^{(n)}(z)}{n!}\right|_{z=\mathbf{0}}
\end{equation*}
that is also an orthogonal series in $L^{2}(\mathbb{B}_{2}^{+};\IC)\cap\ker\bar{\partial}$. Again,
this naturally results from the special structural properties of the basis functions $z^{n}$ which
are, in particular with regard to the Taylor series, the orthogonality and that for each basis
element the condition
$z^{n}\in\left(\partial^{n+1}_{z}\setminus\partial_{z}^{n}\right)\cap\ker\bar{\partial}$ holds. A
direct consequence of these orthogonal power series expansions is that the Fourier- and Taylor
coefficients of a given function $f\in L^{2}(\mathbb{B}_{2}^{+};\IC)\cap\ker\bar{\partial}$ are
explicitly (one-to-one) linked with each other.
\begin{figure}[htb]
\begin{center}
\includegraphics[scale=.5,angle=0]{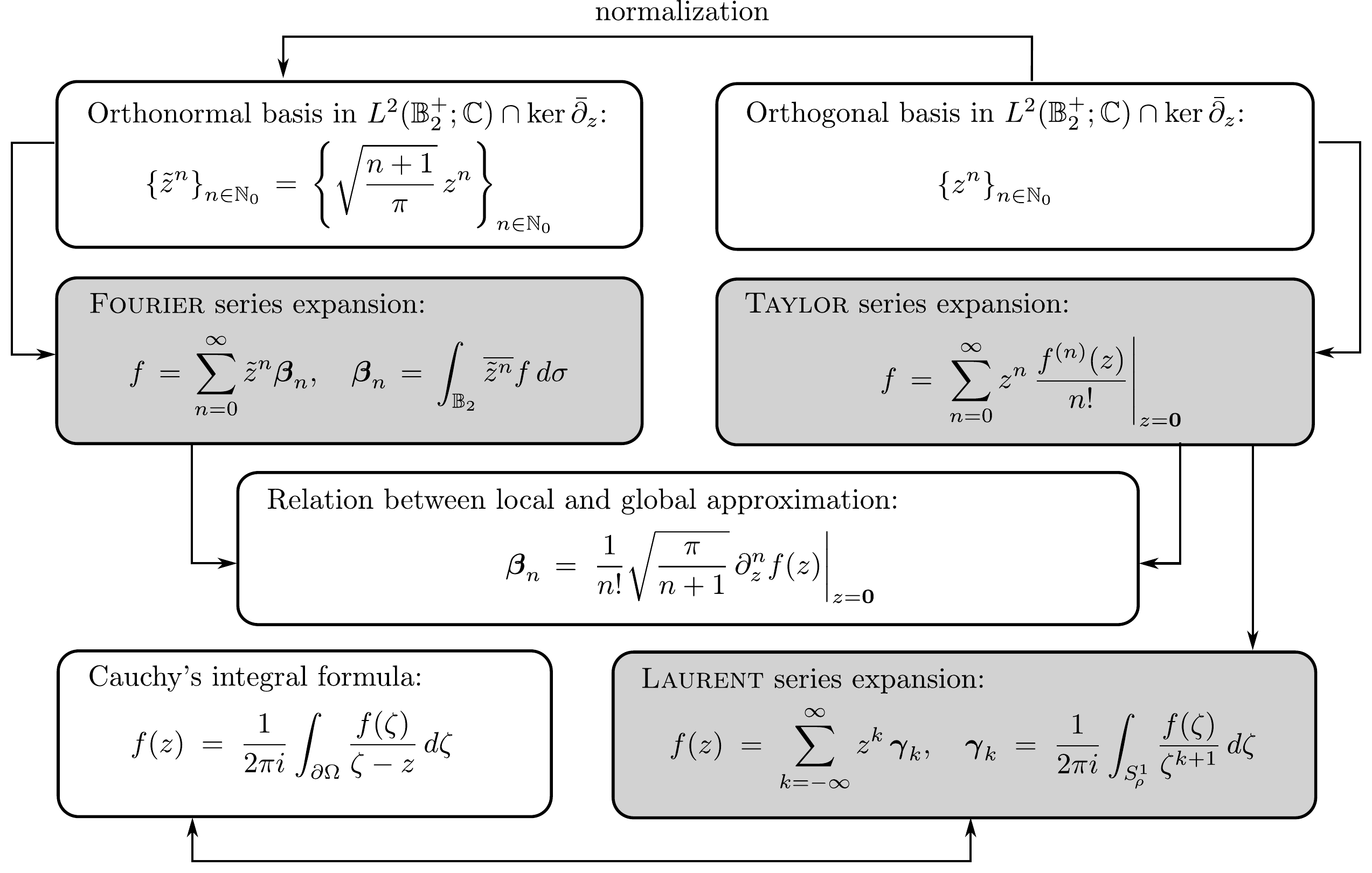}
\caption{ Power and Laurent series expansions in $\IC$. } \label{Figure::OverviewSeriesC}
\end{center}
\end{figure}
If we consider in addition the $z$-powers of negative order then it is also well known that each
holomorphic function defined on an annulus has a unique Laurent series expansion. The corresponding
Laurent coefficients are determined by line integrals which are a generalization of Cauchy's
integral formula. Another remarkable fact is that the secondary part of the Laurent series of a
holomorphic function defined in a certain annulus coincides with their Taylor series expansions and
thus provides another way to compute the Taylor coefficients.

In conclusion of the last paragraph (see Figure \ref{Figure::OverviewSeriesC}), it becomes clear
that the canonical power series expansions (Fourier, Taylor \& Laurent) in the complex
one-dimensional case have an extraordinary quality and a lot of structural relations among each
other. Against this background, one looks also in the higher-dimensional Euclidian space for
function systems generalizing the special structural properties of the series expansions in $\IC$.
An appropriate theoretical framework that is very similar to the complex one-dimensional case is
provided by the theory of monogenic functions. In this article we primarily consider
quaternion-valued monogenic functions that in the following are simply called monogenic or
$\IH$-holomorphic functions, respectively.  Usually (see e.g. \cite{BDS}) monogenic functions are
introduced as solutions of a Dirac equation or of a generalized Cauchy-Riemann system. Already in
1935 it was proved by Fueter \cite{Fueter1935} that a monogenic function can locally be developed
in a Taylor series with respect to the so-called symmetric powers (symmetric products of the
Fueter-variables). This local representation was also used in \cite{BDS} for the description of
monogenic functions and as a technical tool in many proofs. In 1990 it was proved by Malonek
\cite{Mal1990} that monogenic functions can be defined equivalently by their Taylor series with
respect to the symmetric powers and also by a hypercomplex derivability. The latter was refined and
extended to arbitrary dimensions in \cite{Gurlebeck1999} and it was shown that in all real space
dimensions the operator $\frac{1}{2}\partial$ can be understood as a hypercomplex derivative. Here,
$\partial$ stands for the adjoint generalized Cauchy-Riemann operator. A survey on this work and an
extensive list of references can be found in \cite{Gurlebeck2008}. However, in comparison to the
complex Taylor series, where the series itself as well as their complex derivative and primitive
have the property of an orthogonal series, it turns out that the Taylor series in terms of the
Fueter polynomials doesn't generalize these properties. In \cite{AbulConstales} and subsequent
papers by the same authors a system of special monogenic functions was studied. This approach also
allows to define an orthogonal Taylor-type series expansion but was not related to the hypercomplex
derivability. Considering the second branch of monogenic Fourier series, it must be mentioned that
there are a lot of contributions on the global approximation of monogenic functions by orthogonal
polynomials but mostly also not focused on the connection between local and global approximation.
Therefore, a detailed reference list is omitted because the results are not needed here. An
exception takes the constructive approach \cite{DissCacao, Cacao2004a} where a few structural
properties of the complex Fourier series expansion could be generalized. Therein, one could obtain,
similar to the complex case, explicit series representations of the hypercomplex derivative and
primitive based on a monogenic Fourier series expansion in terms of solid spherical monogenics. But
once again, the orthogonality of the resulting series representations (see, e.g., \cite{Cacao2006,
Cacao2007}) isn't preserved. In \cite{DissCacao, GueMorais2007} one could further prove explicit
relations between the coefficients of the aforementioned Fourier series and the Taylor series in
terms of the symmetric powers which yields, due to the missing orthogonality, very extensive
expressions.

Coming again back to the complex case, another canonical series expansion is the Laurent series in
terms of the holomorphic powers $z^{k}$, $k\in\mathbb{Z}$ (see, Figure
\ref{Figure::OverviewSeriesC}). Due to the fact that the complex Laurent series is also based on
the $z$-monomials (extended by the elements of negativ order), the series possesses a direct
interplay with the Taylor series expansion as well as with Cauchy's integral formula. In the higher
dimensional case the situation is again different. Here, the Laurent series expansions so far were
primarily used as a theoretical tool in many proofs (see, e.g., \cite{BDS, Delanghe1992}) whereas
explicit formulae were not readily available. A first constructive attempt was made by van Lancker
\cite{VanLancker1999} using Gegenbauer functions that leads to (integral) representations for a
Taylor-type as well as a Laurent-type series expansion.

In summary and against the background of this article, it can be said that the theory of monogenic
functions provides formally seen all the powerful tools known from the complex one-dimensional
theory, as for instance the mentioned Taylor- and Fourier series expansion as well as a Laurent
series expansion and Cauchy's integral formula (see, e.g., \cite{BDS}). However, from a structural
point of view and particularly with regard to the practical applicability of the aforementioned
function theoretical tools it can be observed that their properties considerably differ from their
complex analogues.

In this contribution a very recent approach of monogenic power and Laurent series expansions in the
framework of real quaternions is presented which generalizes all of the aforementioned properties
of the holomorphic series expansions in $\IC$. These results are part of the thesis \cite{Bock2009}
and summarized in the article on hand. For this reason, the article is structured as follows: In
Section 2 we briefly introduce the algebraic setting and fix some notations. Section 3 deals then
with the generalization of the canonical power series expansions Fourier \& Taylor to higher
dimensions using orthogonal bases of solid spherical monogenics. The resulting series expansions
possess special properties regarding the hypercomplex derivation and primitivation. Here, it should
be emphasized that the series expansions as well as their structural properties were already
studied in \cite{Bock2009a} and thus only summarized here. Details and related proofs may be found
in \cite{Bock2009,Bock2009a}. In section 4 \& 5, some new representation formulae (recurrence,
closed-form) for the basis elements are presented which provide a very compact formulation of the
orthogonal bases. In Section 6 an orthonormal basis of outer solid spherical monogenics in the
space of square integrable quaternion-valued functions is constructed by applying the Kelvin
transformation to the elements of the orthonormal basis studied in Section 3. This orthonormal
basis is then used in Section 7 to construct an orthogonal Appell basis of outer spherical
monogenics and in Section 8 to define an orthogonal Laurent series expansion in the spherical
shell. Finally, the interplay of the new Laurent series expansion with Cauchy's integral formula
and the corresponding Taylor series expansion will be emphasized.
%%%%%%%%%%%%%%%%%%%%%%%%%%%%%%%%%%%%%%%%%%%%%%%%%%%%%%%%%%%%%%%%%%%%%%%%%
%% end of section Introduction
%%%%%%%%%%%%%%%%%%%%%%%%%%%%%%%%%%%%%%%%%%%%%%%%%%%%%%%%%%%%%%%%%%%%%%%%%

%%***********************************************************************
%%***********************************************************************
\section{Preliminaries and Notations}
%%***********************************************************************
%%***********************************************************************
Let $\mathbb{H}$ be the algebra of real quaternions with the standard basis
$\{\textbf{e}_0,\textbf{e}_1,\textbf{e}_2,\textbf{e}_3\}$ subjected to the multiplication rules
\begin{equation*}
\begin{array}{rcl}
\textbf{e}_i \textbf{e}_j+\textbf{e}_j \textbf{e}_i & = & -2 \delta_{ij}\,\textbf{e}_0,\;
i,j=1,2,3,\\
 \mathbf{e}_{1}\mathbf{e}_{2} & = & \mathbf{e}_{3},\;\mathbf{e}_{0}\mathbf{e}_{i} \; = \;
 \mathbf{e}_{i}\mathbf{e}_{0}\;=\; \mathbf{e}_{i},\; i=0,1,2,3.
\end{array}
\end{equation*}
The real vector space $\mathbb{R}^{4}$ will be embedded in $\mathbb{H}$ by identifying the element
$\mathbf{a}=(a_{0},a_{1},a_{2},a_{3})^{\mathrm{T}}\in \mathbb{R}^{4}$ with the quaternion
$\mathbf{a}=a_{0}+a_{1}\mathbf{e}_{1}+a_{2}\mathbf{e}_{2}+a_{3}\mathbf{e}_{3}$,
$a_{i}\in\mathbb{R}$, $i=0,1,2,3$, where $\textbf{e}_{0}=(1,0,0,0)^{\mathrm{T}}$ is the
multiplicative unit element of the algebra $\mathbb{H}$ and will be omitted in expressions if there
is no source for misunderstandings. Further, we denote by
\begin{itemize}
\item[(i)] $\mathbf{Sc}(\mathbf{a})= a_{0}$ the scalar part, $\mathbf{Vec}(\mathbf{a})= \underline{\mathbf{a}} =
\sum_{i=1}^{3}a_{i}\mathbf{e}_{i}$ the vector part of $\mathbf{a}$,
%%%%%%%%%%%%%%%%%%%%%%%
\item[(ii)] $\bar{\mathbf{a}}= a_{0}-\underline{\mathbf{a}}$ the conjugate of
$\mathbf{a}$,
%%%%%%%%%%%%%%%%%%%%%%%
\item[(iii)] $|\mathbf{a}|=\sqrt{\mathbf{a}\,\bar{\mathbf{a}}}$ the norm of $\mathbf{a}$,
%%%%%%%%%%%%%%%%%%%%%%%
\item[(iv)] $\mathbf{a}^{-1} = \frac{\bar{\mathbf{a}}}{|\mathbf{a}|^{2}}$, $\mathbf{a}\neq 0$ the inverse of $\mathbf{a}$.
\end{itemize}

The real vector space $\mathbb{R}^{3}$ will be embedded in $\mathbb{H}$ by the identification of
$\mathbf{x}=(x_{0},x_{1},x_{2})^{\mathrm{T}} \in \mathbb{R}^{3}$ with the \textit{reduced
quaternion} $\mathbf{x}=x_{0}+x_{1}\textbf{e}_{1}+x_{2}\textbf{e}_{2}$. As a consequence, we will
often use the same symbol $\mathbf{x}$ to represent a point in $\mathbb{R}^{3}$ as well as to
represent the corresponding reduced quaternion.

Let now $\Omega$ be an open subset of $\mathbb{R}^{3}$ with a piecewise smooth boundary. An
$\mathbb{H}$-valued function is a mapping $f:\Omega\longrightarrow \mathbb{H}$ such that
$f(\mathbf{x})=\sum_{i=0}^{3} f^{i}(\mathbf{x})\,\textbf{e}_{i}\,,\;\mathbf{x} \in \Omega$. The
coordinates $f^{i}(\mathbf{x})$ are real-valued functions defined in $\Omega$, i.e.,
$f^{i}(\mathbf{x}):\Omega \longrightarrow \mathbb{R}\,,i=0,1,2,3$. Continuity, differentiability or
integrability of $f$ are defined coordinate-wisely. We will work with the right $\mathbb{H}$-linear
Hilbert space of square-integrable $\mathbb{H}$-valued functions in $\Omega$ that is denoted by
$L^{2}(\Omega;\mathbb{H})$ and equipped with the $\mathbb{H}$-valued inner product
\begin{equation*}
<\,f,g\,>_{L^{2}(\sOm;\sIH)}= \int_{\Omega}\,\bar{f}\,g\,dV.
\end{equation*}
Here $dV$ denotes the Lebesgue measure in $\mathbb{R}^{3}$. Furthermore, the operator
\begin{equation}\label{Equation::CROperator}
\bar{\partial} := \frac{\partial}{\partial x_{0}} + \frac{\partial}{\partial
\underline{\mathbf{x}}} \;=\; \frac{\partial}{\partial
x_{0}}+\textbf{e}_{1}\,\frac{\partial}{\partial x_{1}}+\textbf{e}_{2}\,\frac{\partial}{\partial
x_{2}}
\end{equation}
is called {\it generalized Cauchy-Riemann} operator. The corresponding {\it adjoint generalized
Cauchy-Riemann} operator is defined by
\begin{equation}\label{Equation::conjugateCROperator}
\partial := \frac{\partial}{\partial
x_{0}} - \frac{\partial}{\partial \underline{\mathbf{x}}} \;=\; \frac{\partial}{\partial
x_{0}}-\textbf{e}_{1}\,\frac{\partial}{\partial x_{1}}-\textbf{e}_{2}\,\frac{\partial}{\partial
x_{2}}.
\end{equation}
At this point, we emphasize that throughout this article the introduced differential operators are
considered as operators acting from the left and analogously denoted as in the complex analysis
(see, e.g., \cite{Gurlebeck2008}) which is vice versa to the commonly used notation in Clifford
analysis. This leads to the following definitions:
\begin{definition}
A function $f\in C^{1}(\Omega;\mathbb{H})$ is called {\it monogenic} or {\it$\IH$-holomorphic} in
$\Omega\subset\IR^3$ if
\begin{equation}\label{Equation::MonogenicFunction}
\bar{\partial} f=0\;\;\mbox{in}\;\;\Omega\;\;(\mbox{or
equivalently}\;\;f\in\ker\bar{\partial}\;\;\mbox{in}\;\;\Omega).
\end{equation}
Conversely, a function $\widehat{f}\in C^{1}(\Omega;\mathbb{H})$ is called {\it anti-monogenic} or
{\it anti-$\IH$-holomorphic} in $\Omega\subset\IR^3$ if
\begin{equation}\label{Equation::Anti-MonogenicFunction}
\partial\widehat{f} = 0\;\;\mbox{in}\;\;\Omega\;\;(\mbox{or
equivalently}\;\;f\in\ker\partial\;\;\mbox{in}\;\;\Omega).
\end{equation}
\end{definition}
In view of the upcoming calculations we also need a more specific formulation of the monogenicity
criteria (\ref{Equation::MonogenicFunction}) that is in particular advantageous for the conversion
of a given monogenic function into an anti-monogenic function and vice versa. Thus, we introduce
for a given function $f\in C^{1}(\Omega;\IH)$ the notation $f = f_{\mathbf{e}_{0}\mathbf{e}_{3}} +
f_{\mathbf{e}_{1}\mathbf{e}_{2}} := \left(f^{0}+f^{3}\mathbf{e}_{3}\right) +
\left(f^{1}\mathbf{e}_{1}+f^{2}\mathbf{e}_{2}\right)$. Applying now the compact formulation of the
differential operator (\ref{Equation::CROperator}) to $f$ yields an equivalent definition of
$\IH$-holomorphy (\ref{Equation::MonogenicFunction}) given by the system
\begin{equation}\label{Equation::MonogenicFunction_PDESystem}
\left.
\begin{array}{lcrcl}
f \in \ker \bar{\partial} & \Leftrightarrow & \D\frac{\partial
f_{\mathbf{e}_{0}\mathbf{e}_{3}}}{\partial x_{0}} + \frac{\partial
f_{\mathbf{e}_{1}\mathbf{e}_{2}}}{\partial
\underline{\mathbf{x}}}& = & 0,\\
%%%%%%%%%%%%%%%%%%%%%%%%%%%%%%%%%%%%%%%%%%%%%%%%
& \Leftrightarrow &\D\frac{\partial f_{\mathbf{e}_{1}\mathbf{e}_{2}}}{\partial x_{0}} +
\frac{\partial f_{\mathbf{e}_{0}\mathbf{e}_{3}}}{\partial \underline{\mathbf{x}}}& = & 0.
\end{array}\right\}
\end{equation}
Based on the last representation one can conclude a significant connection between monogenic and
anti-monogenic functions.
\begin{corollary}\label{Corollary::Anti-MonogenicFunction-Construction}
Let $f=\sum_{i=0}^{3}f^{i}\mathbf{e}_{i}\,\in C^{1}(\Omega;\IH)$ be a monogenic function in
$\Omega\subset\IR^{3}$. The function
\begin{equation}\label{Equation::Anti-MonogenicFunction-Construction}
\widehat{f}\;:=\; f^{0} - f^{1}\mathbf{e}_{1} - f^{2}\mathbf{e}_{2} + f^{3}\mathbf{e}_{3}
\end{equation}
defines an anti-monogenic function in $\Omega$, such that $\partial \widehat{f} = 0$.
\end{corollary}
\begin{proof}
Let $f\in\ker\bar\partial$ be denoted by $f = f_{\mathbf{e}_{0}\mathbf{e}_{3}} +
f_{\mathbf{e}_{1}\mathbf{e}_{2}} = \left(f^{0}+f^{3}\mathbf{e}_{3}\right) +
\left(f^{1}\mathbf{e}_{1}+f^{2}\mathbf{e}_{2}\right)$, we have $\widehat{f} =
f_{\mathbf{e}_{0}\mathbf{e}_{3}} - f_{\mathbf{e}_{1}\mathbf{e}_{2}}$. Applying the compact
formulation of the differential operator (\ref{Equation::conjugateCROperator}) and using the
relations (\ref{Equation::MonogenicFunction_PDESystem}) yields the proof of the corollary
\begin{equation*}
0\;\stackrel{!}{=}\;\partial \widehat{f}\,=\, \left(\frac{\partial
f_{\mathbf{e}_{0}\mathbf{e}_{3}}}{\partial x_{0}} + \frac{\partial
f_{\mathbf{e}_{1}\mathbf{e}_{2}}}{\partial \underline{\mathbf{x}}}\right) - \left(\frac{\partial
f_{\mathbf{e}_{1}\mathbf{e}_{2}}}{\partial x_{0}} + \frac{\partial
f_{\mathbf{e}_{0}\mathbf{e}_{3}}}{\partial \underline{\mathbf{x}}}\right)\,=\,0.
\end{equation*}
\end{proof}

Here, it should be emphasized that in the complex one-dimensional case the conjugation of an
holomorphic function $f\in C^{1}(\Omega;\IC)$ gives directly the corresponding anti-holomorphic
function $\widehat{f}$ and thus $\bar{f} \equiv \widehat{f}$. For $\IH$-valued monogenic functions
this property doesn't hold in general as Corollary
\ref{Corollary::Anti-MonogenicFunction-Construction} shows. An exceptional case takes the subset of
$\mathcal{A}$-valued monogenic functions which can be easily deduced from the latter calculations.

Furthermore, we need the well known concept of the hypercomplex derivative that was first
introduced by Malonek \cite{Malonek1987} for the quaternionic case and later on also generalized to
arbitrary dimensions \cite{Gurlebeck1999}.
\begin{definition}[Hypercomplex derivative]\label{Definition::HypercomplexDerivative}
Let $f\in C^{1}(\Omega;\IH)$ be a continuous, real-differentiable function and monogenic in
$\Omega$. The expression $\partial_{0}f := \frac{1}{2}\partial f$ is called {\it hypercomplex
derivative} of $f$ in $\Omega$.
\end{definition}
As a consequence of Definition \ref{Definition::HypercomplexDerivative}, we introduce a special
subset of monogenic functions characterized by vanishing first derivatives.
\begin{definition}[Monogenic constant]\label{Definition::MonogenicConstant}
A $C^{1}$-function belonging to $\ker \partial_{0}\cap\ker\bar{\partial}$ is called monogenic
constant.
\end{definition}
Moreover, we state the definition of a monogenic primitive.
\begin{definition}[Primitive]
A function $F\in C^{1}(\Omega;\IH)$ is called monogenic primitive of a monogenic function $f$ with
respect to the hypercomplex derivative $\partial_{0}$, if
\begin{equation*}
F \in \ker \bar{\partial}\quad \text{ and } \quad \partial_{0} F = f.
\end{equation*}
If for a given function $f\in\ker\bar{\partial}$ such a function $F$ exists then it will be denoted
by $\mathcal{P}f := F$.
\end{definition}
Let us remark that in the complex theory a holomorphic primitive of a holomorphic function is
simply constructed by line integrals. However, a higher dimensional analogue of this elementary
procedure does not exist due to line integrals are in general not path independent. Here, we will
use an operator approach as for instance applied in \cite{Cacao2007}, where an primitivation
operator acting on each element of an orthogonal basis is defined and extended by continuity to the
whole space. In this connection, it is important to mention that already several definitions for
the monogenic primitive as a right inverse operator to the hypercomplex derivative $\partial_{0}$
exist (see, e.g., \cite{Delanghe2006, GueMorais2007} or a survey in \cite{Gurlebeck2008}). The
primitivation operator defined in this article extend the results of \cite{Cacao2007} by important
structural properties.
%%%%%%%%%%%%%%%%%%%%%%%%%%%%%%%%%%%%%%%%%%%%%%%%%%%%%%%%%%%%%%%%%%%
%% end of section Preliminaries and Notations
%%%%%%%%%%%%%%%%%%%%%%%%%%%%%%%%%%%%%%%%%%%%%%%%%%%%%%%%%%%%%%%%%%%

%%************************************************************************************************
%%************************************************************************************************
\section{Orthonormal bases of solid spherical monogenics and power series expansions}
%%************************************************************************************************
%%************************************************************************************************

As it was already emphasized in the beginning of this paper, complex power series expansions are in
particular orthogonal series expansions with respect to the unit disc. It is for this reason to
look now in the spatial case for orthogonal systems on the first place with respect to the unit
ball in $\IR^3$. First of all, let us fix some notations:
\begin{definition}
\begin{itemize}
\item[(\textsc{i})] Let $\IBp:= \IBp(0,1)$ be the unit ball
and $\overline{\mathbb{B}}_{3}^{+} := \IBp \cup S^{2} $ its closure with the unit sphere
$S^{2}:=\partial\IBp$. The corresponding outer domain is denoted by $\IBm:= \IR^{3} \setminus
\overline{\mathbb{B}}_{3}^{+}$ and its closure by $\overline{\mathbb{B}}_{3}^{-}:= \IR^{3}
\setminus \IBp$.
%%%%%%%%%%%%%%%%%%%
\item[(\textsc{ii})] Let $\mathcal{M}_{n}^{+}$ be the space of homogeneous monogenic polynomials of degree $n$.
An arbitrary element $P_{n}$ of $\mathcal{M}_{n}^{+}$ is called \textit{inner solid spherical
monogenic} of degree $n$.
%%%%%%%%%%%%%%%%%%%
\item[(\textsc{iii})]  Let $\mathcal{M}_{n}^{-}$ be the space of homogeneous monogenic functions in
$\IR_{0}^{3} = \IR^{3}\setminus\{0\}$ with degree of homogeneity $-(n+2)$. An arbitrary element
$Q_{n}$ of $\mathcal{M}_{n}^{-}$ is called \textit{outer solid spherical monogenic} of degree $n$.
\end{itemize}
\end{definition}
As a matter of course, we will work with spherical coordinates
\begin{equation*}
x_0 = r \cos \theta,\quad x_1 = r \sin \theta \cos\varphi,\quad x_2 = r \sin \theta \sin \varphi,
\end{equation*}
where $0<r<\infty$, $0<\theta\le\pi$, $0<\varphi\le 2\pi$. Each $\mathbf{x}\in \IR^3\setminus
\{\mathbf{0}\}$ admits a unique representation $\mathbf{x} = r\boldsymbol{\omega}$, where
$\boldsymbol{\omega}=\omega_0+\omega_1\mathbf{e}_1+\omega_2\mathbf{e}_2$, with
$\omega_j=\frac{x_j}{r}$ ($j=0,1,2$) and $|\boldsymbol{\omega}|=1$. The spherical representation of
the hypercomplex derivative is then given by
\begin{equation}\label{Equation::conjugateCROperator-spherical}
\partial_{0} = \frac{1}{2}\left( \overline{\boldsymbol{\omega}}\frac{\partial}{\partial
r}+\frac{1}{r}\;\overline{L}\right)
\end{equation}
where
\begin{equation*}
\overline{L}\,=\,(-\sin{\theta}-\cos{\theta}\cos{\varphi}\,\textbf{e}_1-
\cos{\theta}\sin{\varphi}\,\textbf{e}_2)\frac{\partial}{\partial
\theta}+\frac{1}{\sin{\theta}}(\sin{\varphi}\,\textbf{e}_1-\cos{\varphi}\,\textbf{e}_2)\frac{\partial}{\partial
\varphi}.
\end{equation*}

%%%%%%%%%%%%%%%%%%%%%%%%%%%%%%%%%%%%%%%%%%%%%%%%%%%%%%%%%%
\subsection{The basic polynomial toolbox}
%%%%%%%%%%%%%%%%%%%%%%%%%%%%%%%%%%%%%%%%%%%%%%%%%%%%%%%%%%

One simple and constructive way to generate monogenic functions is based on the factorization of
the Laplace operator $\Delta = \partial\bar{\partial}$ by the generalized Cauchy-Riemann operator
and its adjoint operator. In other words this means that applying the operator $\partial$ to a
harmonic function subsequently yields a monogenic function. In \cite{DissCacao, Cacao2004a} this
construction principle was used for the first time to construct a linearly independent system of
inner solid spherical monogenics in a quite explicit way. This construction was based on the well
known system of spherical harmonics
\begin{equation*}
\{U_{n+1}^l,\;V_{n+1}^m,\,\,l=0,\ldots,n+1,\,m=1,\ldots,n+1 \}_{n \in \INo}
\end{equation*}
considered, e.g., in \cite{Sansone1959}. Here, we will only summarize the final representation
formulae of the system elements and refer for a detailed study and related proofs to
\cite{DissCacao, Cacao2004a}. The set of inner solid spherical monogenics is defined by
\begin{equation}\label{Equation::SolidSpherMonogenics}
\{r^{n}\,X_{n}^l, \;r^{n}\,Y_{n}^m,\,\,l=0,\ldots,n+1,m=1,\ldots,n+1 \}_{n \in \INo}
\end{equation}
where the so-called spherical monogenics are given by
\begin{eqnarray}
X_{n}^l &:=&\partial_{0}\left[r^{n+1}\;U_{n+1}^l\right]\Bigl|_{r=1}\nonumber\\
      &=&\A^{l,{n}} \cos l{\varphi}\nonumber\\
      & &+\,(\B^{l,{n}}\cos{\varphi}\cos{l\varphi}-\C^{l,{n}} \sin{\varphi}\sin{l\varphi})\,\textbf{e}_1\nonumber\\
      & &+\,(\B^{l,{n}}\sin{\varphi}\cos{l\varphi}+\C^{l,{n}}\cos{\varphi}\sin{l\varphi})\,\textbf{e}_2,
 \label{Equation::MonogenicPoly_Xm}
\end{eqnarray}
\begin{eqnarray}
Y_{n}^m &:=&\partial_{0}\left[r^{n+1}\;V_{n+1}^m\right]\Bigl|_{r=1}\nonumber\\
      &=&\A^{m,{n}} \sin m\varphi\nonumber\\
      & &+\,(\B^{m,{n}}\cos{\varphi}\sin{m\varphi}+\C^{m,{n}} \sin{\varphi}\cos{m\varphi})\,\textbf{e}_1\nonumber\\
      & &+\,(\B^{m,{n}}\sin{\varphi}\sin{m\varphi}-\C^{m,{n}} \cos{\varphi}\cos{m\varphi})\,\textbf{e}_2,
      \label{Equation::MonogenicPoly_Ym}
\end{eqnarray}
\noindent with the coefficient functions
\begin{eqnarray}
\A^{m,{n}}&\hspace{-0.3cm}:=&\hspace{-0.3cm}\frac{1}{2}\Bigl(\sin^{2}\theta\frac{d}{dt}[P_{n+1}^m(t)]_{t=\cos{\theta}}+
(n+1)\cos{\theta}P_{n+1}^m(\cos{\theta})\Bigr),\label{Equation::HomogMonogPoly_Am}\\
%%%%%%%%%%%%%%%%%%%%%%%%%%%%%%%%%%%%%%
\B^{m,{n}}&\hspace{-0.3cm}:=&\hspace{-0.3cm}\frac{1}{2}\Bigl(\sin{\theta}\cos{\theta}
\frac{d}{dt}[P_{n+1}^m(t)]_{t=\cos{\theta}}-(n+1)\sin{\theta} P_{n+1}^m(\cos{\theta})\Bigr),
\label{Equation::HomogMonogPoly_Bm}\\
%%%%%%%%%%%%%%%%%%%%%%%%%%%%%%%%%
\C^{m,{n}}&\hspace{-0.3cm}:=&\hspace{-0.3cm}\frac{1}{2}\,m
\frac{1}{\sin{\theta}}\,P_{n+1}^m(\cos{\theta}).\label{Equation::HomogMonogPoly_Cm}
\end{eqnarray}
The functions $P_{n+1}^m$ appearing in
(\ref{Equation::HomogMonogPoly_Am})-(\ref{Equation::HomogMonogPoly_Cm}) are the associated Legendre
functions which, as it is well known, are solutions to several recurrence formulae. In this article
we particularly need the recurrence relations
\begin{eqnarray}
(1-t^2)(P_{n+1}^m(t))' & = & (n+m+1) P_{n}^{m}(t) -
(n+1)\,t\,P_{n+1}^{m}(t),\label{Equation::LegendreRecurrence_I}\\[2ex]
%%%%%%%%%%%%%%%%%%%%
(1-t^2)^{1/2}(P_{n+1}^m(t))' & = & P_{n+1}^{m+1}(t)-m\,(1-t^2)^{-1/2}\,t
\,P_{n+1}^m(t),\label{Equation::LegendreRecurrence_II}\\[2ex]
%%%%%%%%%%%%%%%%%%%%
(1-t^2)^{1/2}P_{n+1}^m(t)& = &
\frac{1}{2n+3}\left(P_{n+2}^{m+1}(t)-P_{n}^{m+1}(t)\right)\label{Equation::LegendreRecurrence_III}
\end{eqnarray}
and the two-step formula
\begin{equation}\label{Equation::LegendreRecurrence_ThreeTermFormula}
(n-m+1)P_{n+1}^m(t)-(2n+1)\,t \,P_{n}^m(t)+(n+m)P_{n-1}^m(t)=0,
\end{equation}
where $m=0,\ldots,n+1$. For a detailed study of the associated Legendre functions we refer, e.g.,
to \cite{Andrews1998} and \cite{Sansone1959}. Here, it must be remarked that the system of solid
spherical monogenics (\ref{Equation::SolidSpherMonogenics}) is also extensively studied in the
context of square integrable $\mathcal{A}$-valued functions (solutions to the so-called Riesz
system). Already in \cite{DissCacao} it was proved that system
(\ref{Equation::SolidSpherMonogenics}) is an orthogonal system with respect to the unit ball $\IBp$
and in \cite{Cacao-Malonek2006} its completeness in
$L^{2}(\IBp;\mathcal{A})\cap\ker\bar{\partial}$. Later on, this system was used in
\cite{GueMorais2007,GueMorais2009,GueMorais2010a,GueMorais2010b} to prove real part estimates of
monogenic power series and some further structural properties of the polynomials itself. The reader
is cautioned that the set $\mathcal{A}$ is only a real vector space but not a sub-algebra of $\IH$
and thus considerably differs from the situation considered in this article.

Finally, we end up by recalling a result that relates the coefficient functions
(\ref{Equation::HomogMonogPoly_Am})-(\ref{Equation::HomogMonogPoly_Cm}) among each other. In
\cite[proof of Proposition 3.1]{Bock2007b} it was proved that for a fixed $n\in\INo$ the relations
\begin{eqnarray}
(n+m+2)\,\A^{m,n} & = & \C^{m+1,n} - \B^{m+1,n},\label{Equation::RelationABC_I}\\
%%%%%%%%%%%%%%%%%%%%%%%%%%%%%%%%
\A^{m+1,n} & = & (n+m+2)\bigl(\B^{m,n} + \C^{m,n}\bigr)\label{Equation::RelationABC_II}
\end{eqnarray}
hold, with $m=0,\ldots,n$.

%%%%%%%%%%%%%%%%%%%%%%%%%%%%%%%%%%%%%%%%%%%%%%%%%%%%%%%%%%%%%%%%%%%%%%%%%%%%%%%%%%%%%%%%%%%%%
\subsection{An $\IH$-linear complete orthonormal system of inner solid spherical monogenics}
%%%%%%%%%%%%%%%%%%%%%%%%%%%%%%%%%%%%%%%%%%%%%%%%%%%%%%%%%%%%%%%%%%%%%%%%%%%%%%%%%%%%%%%%%%%%%%

For each $n \in \INo$, we now denote the normalized set of inner solid spherical monogenics
(\ref{Equation::SolidSpherMonogenics}) multiplied by a basis element $\mathbf{e}_{j}$, $j=0,1,2,3$
from the right by
\begin{equation}\label{Equation::SolidSpherMonogenics-nomalized}
\tilde{X}_{n,j}^{0,\dagger}:=\frac{r^{n}X_n^{0}\,\mathbf{e}_{j}}{\bigl\|r^{n}X_n^{0}\bigr\|_{L^{2}(\sIBp)}},\quad
%%%%%%%%%%%%%%%%%%%%%%%
\tilde{X}_{n,j}^{m,\dagger}:=
\frac{r^{n}X_n^{m}\,\mathbf{e}_{j}}{\bigl\|r^{n}X_n^{m}\bigr\|_{L^{2}(\sIBp)}},\quad
%%%%%%%%%%%%%%%%%%%%%%%
\tilde{Y}_{n,j}^{m,\dagger}:=
\frac{r^{n}Y_n^{m}\,\mathbf{e}_{j}}{\bigl\|r^{n}Y_n^{m}\bigr\|_{L^{2}(\sIBp)}},
\end{equation}
where the norms (see, e.g., \cite{DissCacao}) are explicitly given by
\begin{equation}\label{Equation::NormSolidSpherMonogenics}
\left.
\begin{array}{l}
\D\bigl\|r^{n}X_{n}^{0}\bigr\|_{L^{2}(\sIBp;\sIH)}\;=\;\sqrt{\frac{\pi(n+1)}{2n+3}},\\
%%%%%%%%%%%%%%%%%%%%%%%%%%%%%%%%%%%%%
\D\bigl\|r^{n}X_{n}^{m}\bigr\|_{L^{2}(\sIBp;\sIH)}\;=\;\bigl\|r^{n}Y_{n}^{m}\bigr\|_{L^{2}(\sIBp;\sIH)}\;=\;
\sqrt{\frac{\pi(n+1)(n+m+1)!}{2(2n+3)(n-m+1)!}},
\end{array}\right\}
\end{equation}
with $m=1,\ldots,n+1$. As it was already stated in \cite{Bock2008} and particularly proved in
\cite{Bock2009,Bock2009a}, special $\IH$-linear combinations of the inner solid spherical
monogenics (\ref{Equation::SolidSpherMonogenics-nomalized}) constitute the orthonormal system:
\begin{theorem}[\cite{Bock2009,Bock2009a}]
For each $n\in \INo$ the following $n+1$ inner solid spherical monogenics are orthonormal in
$L^{2}(\IBp;\IH) \cap \ker \bar{\partial}$:
\begin{equation}\label{Equation::CONS_IH}
\left.
\begin{array}{lcl}
\varphi_{n,\sIH}^{0} & := & \tilde{X}_{n,0}^{0,\dagger},\vspace{0.1cm}\\
%%%%%%%%%%%%%%%
\varphi_{n,\sIH}^{l} & := & c_{n,-l}\left( \tilde{X}_{n,0}^{l,\dagger} -
\tilde{Y}_{n,3}^{l,\dagger}\right),
\end{array}\right\}
\end{equation}
where $c_{n,-l}=\sqrt{\frac{n+1}{2(n-l+1)}}$ and $l=1,\ldots,n$.
\end{theorem}
In \cite{Sudbery1979}, Sudbery proved that $\dim \mathcal{M}_{n}^{+}(\IH) = n+1$. It is also known
that for $n\neq m$ the elements of the subspaces $\mathcal{M}_{n}^{+}(\IH)$ and
$\mathcal{M}_{m}^{+}(\IH)$ are automatically orthogonal in $L^{2}(\IBp;\IH) \cap \ker
\bar{\partial}$. Since, for a fixed $n\in\INo$, the system (\ref{Equation::CONS_IH}) is of
dimension $n+1$ and thus an orthonormal basis in $\mathcal{M}_{n}^{+}(\IH) \subset L^{2}(\IBp;\IH)
\cap \ker \bar{\partial}$, we conclude:
\begin{corollary}[\cite{Bock2009,Bock2009a}]
The system of inner solid spherical monogenics $\bigl\{
\varphi_{n,\sIH}^{l}\,:\,l=0,\ldots,n\bigr\}_{n\in\INo}$ is an orthonormal basis in
$L^{2}(\IBp;\IH) \cap \ker \bar{\partial}$.
\end{corollary}
Due to the orthogonality and the completeness of the orthonormal system (\ref{Equation::CONS_IH})
we state the Fourier series expansion in $L_{2}(\IBp;\IH) \cap \ker \bar{\partial}$.
\begin{corollary}[Fourier series in $L_{2}(\IBp;\IH) \cap \ker \bar{\partial}$]
Let $f \in L_{2}(\IBp;\IH) \cap \ker \bar{\partial}$. Then $f$ can be uniquely represented in terms
of the orthonormal system (\ref{Equation::CONS_IH}), that is:
\begin{equation}\label{Equation::FourierSeries}
f\,:=\, \sum_{n=0}^{\infty}\sum_{l=0}^{n}\,\varphi_{n,\sIH}^{l}\,
\boldsymbol{\alpha}_{n,l},\quad\text{with}\quad\boldsymbol{\alpha}_{n,l}\,=\,\int_{\mathbb{B}_{3}}
\overline{\varphi_{n,\sIH}^{l}}\,f\,dV.
\end{equation}
\end{corollary}
Here, it should be emphasized that in contrast to the complex case the order of
$\varphi_{n,\sIH}^{l}$ and $f$ in the inner products has to be respected. Equivalently,
characterizing $f\in L_{2}(\IBp;\IH) \cap \ker \bar{\partial}$ by means of the Fourier coefficients
yields naturally \textit{Parsevals identity}:
\begin{corollary}\label{Equation::ParsevalsIdentity}
$f \in L_{2}(\IBp;\IH) \cap \ker \bar{\partial}$ is equivalent to $\D
\sum_{n=0}^{\infty}\sum_{l=0}^{n}\,|\boldsymbol{\alpha}_{n,l}|^{2}\,<\,\infty$.
\end{corollary}

Let us now focus on some important structural properties of the orthonormal system
(\ref{Equation::CONS_IH}) which extend the $\IH$-linear system constructed in
\cite{DissCacao,Cacao2004a} significantly.
\begin{theorem}[\cite{Bock2009,Bock2009a}]\label{Theorem::CONS_IH-Properties}
For the polynomials $\varphi_{n,\sIH}^{l}$, $l=0,\ldots,n$ of system (\ref{Equation::CONS_IH}), the
following properties hold:
\begin{itemize}
\item[(\textsc{i})] Applying the hypercomplex derivative $\partial_{0} =
\frac{1}{2}\partial$ to the complete orthonormal system (\ref{Equation::CONS_IH}) yields
\begin{equation*}
\partial_{0}\,\varphi_{n,\sIH}^{k}\,=\, \sqrt{\frac{(2n+3)(n-k)(n+k+1)}{2n+1}}\,\varphi_{n-1,\sIH}^{k},\,k=0,\ldots,n-1,\;n\in \IN.
\end{equation*}
%%%%%%%%%%%%%%%%%%%%%%%%%%%%%%%%%%%%%%%
\item[(\textsc{ii})] The operator $\mathcal{P}_{\sIH}:\mathcal{M}^{+}_{n}(\IH)
\rightarrow \mathcal{M}^{+}_{n+1}(\IH)$ given by
\begin{equation*}
\mathcal{P}_{\sIH}\,\varphi_{n,\sIH}^{l}\,=\,
\sqrt{\frac{2n+3}{(2n+5)(n-l+1)(n+l+2)}}\,\varphi_{n+1,\sIH}^{l},
\end{equation*}
defines a primitive on each basis element, such that
$\partial_{0}\left[\mathcal{P}_{\sIH}\,\varphi_{n,\sIH}^{l}\right] = \varphi_{n,\sIH}^{l}$ and
$l=0,\ldots,n$, $n\in \INo$.
%%%%%%%%%%%%%%%%%%%%%%%%%%%%%%%%%%%%%%
\item[(\textsc{iii})] Denoting by $\partial_{0}^{n}$ the $n$-fold application of the
hypercomplex derivative, we state
\begin{equation*}
\varphi_{n,\sIH}^{l}\,\in\, \bigl( \ker\partial_{0}^{n-l+1} \setminus \ker\partial_{0}^{n-l}\bigr)
\cap \ker\bar{\partial},\; l=0,\ldots,n,\,n\in \INo,
\end{equation*}
where $\partial_{0}^{0}$ is identified with the identity operator.
\end{itemize}
\end{theorem}
As a direct consequence of the last theorem, Figure \ref{Figure::CONS_IH} qualitatively illustrates
the structural ordering of the basis elements and the mapping properties of the hypercomplex
derivative as well as the primitivation operator.
\begin{figure}[htb]
\begin{center}
\includegraphics[scale=1.5,angle=0]{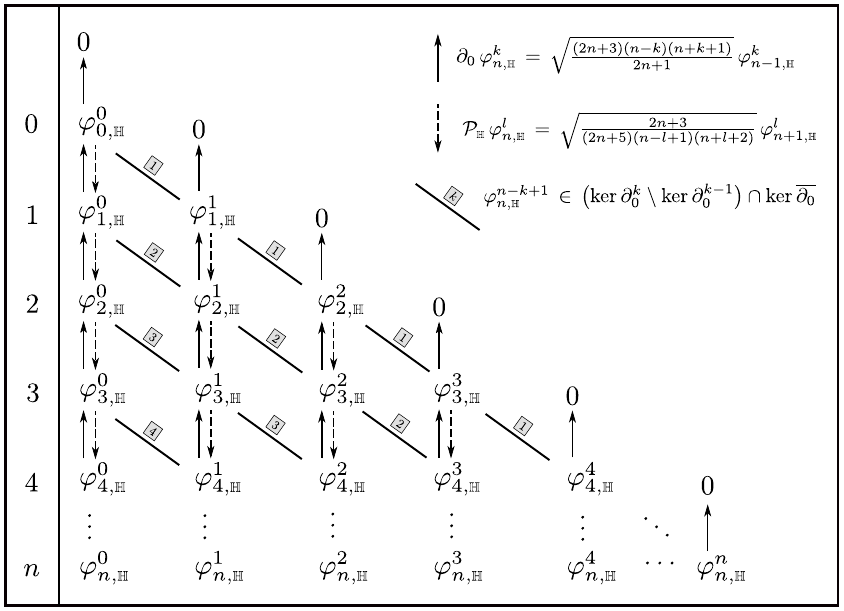}
\caption{Structural scheme of the orthonormal basis in $L^{2}(\IBp;\IH) \cap \ker \bar{\partial}$.}
\label{Figure::CONS_IH}
\end{center}
\end{figure}
Here, some key facts should be highlighted. Firstly, the basis elements are arranged in a lower
triangular matrix in which the polynomials of degree $n$ belong to the $(n+1)$-th row. Secondly,
the application of the hypercomplex derivative $\partial_{0}$ or the primitivation operator
$\mathcal{P}_{\sIH}$ can be viewed as operations on a fixed column of the scheme mapping a basis
element to a single basis element just by multiplying with a real factor and lowering or raising
the polynomial degree. Thirdly, the polynomials $\left\{\varphi_{n,\sIH}^{n} \right\}_{n\in\INo}$
on the upper diagonal are the set of monogenic constants (see Definition
\ref{Definition::MonogenicConstant}) vanishing after applying the hypercomplex derivative once.

%%*******************************************************************************
\subsubsection{Derivatives of $\IH$-holomorphic functions}
%%*******************************************************************************

Let us now consider the complete orthonormal system (\ref{Equation::CONS_IH}) in $L_{2}(\IBp;\IH)
\cap \ker \bar{\partial}$ in particular with respect to the hypercomplex derivative $\partial_{0}$.
In this context, the decisive property of the basis (\ref{Equation::CONS_IH}) is that the
application of the hypercomplex derivative to a polynomial of degree $n$ results again to a real
multiple of a single basis element of degree $n-1$. Due to that, the $\IH$-linear system possesses
a strong analogy to the complex power series expansions that will be elaborated next.

From paragraph (\textsc{i}) of Theorem \ref{Theorem::CONS_IH-Properties} we conclude directly:
\begin{proposition}[\cite{Bock2009,Bock2009a}]
The homogeneous monogenic polynomials $\Bigl\{
\partial_{0}\varphi_{n,\sIH}^{k}:
k=0,\ldots,n-1\Bigr\}_{n\in\IN}$ form an orthogonal basis in $L_{2}(\IBp;\IH) \cap \ker
\bar{\partial}$.
\end{proposition}

The hypercomplex derivative of an arbitrary monogenic function represented by the corresponding
Fourier series in $L_{2}(\IBp;\IH) \cap \ker \bar{\partial}$ is given as follows.

\begin{proposition}[\cite{Bock2009,Bock2009a}]
Let $f\in L_{2}(\IBp;\IH) \cap \ker \bar{\partial}$ be represented by its associated Fourier series
(\ref{Equation::FourierSeries}). Applying $\partial_{0}$ to each summand yields formally the series
\begin{equation}\label{Equation::FourierSeries-Derivative}
\partial_{0}\,f\,=\, \sum_{n=0}^{\infty}\sum_{k=0}^{n} \,\sqrt{\frac{(2n+5)(n-k+1)(n+k+2)}{2n+3}}\,\varphi_{n,\sIH}^{k}\,
\boldsymbol{\alpha}_{n+1,k},
\end{equation}
where $\boldsymbol{\alpha}_{n,l}\in\IH$. This series converges pointwisely in $\IBp$ and in
$L_2(\IBp(0,r);\IH)$ for all $r<1$.
\end{proposition}

Comparing the representation (\ref{Equation::FourierSeries-Derivative}) with the system of
hypercomplex derivatives constructed in \cite{Cacao2006}, we emphasize that the decisive quality of
the new system results from the fact that the hypercomplex derivative of an arbitrary function
represented by its Fourier series yields again an orthogonal series. More precisely, in
\cite{Cacao2006} one could just prove that the operator $\partial_{0}$ maps an orthonormal basis to
a set of linear independent polynomials. Now the resulting set is also an orthogonal basis which is
quite analogous to the complex one-dimensional case. To illustrate the afore said, let us consider,
for instance, the Fourier series of a holomorphic function $f\in L_{2}(\mathbb{B}^{+}_{2};\IC) \cap
\ker \bar{\partial}_{z}$, given by
\begin{equation*}
f\,=\,\sum_{n=0}^{\infty} \tilde{z}^{n}\boldsymbol{\beta}_{n},\quad\boldsymbol{\beta}_{n}\in \IC,
\end{equation*}
where the set $\left\{ \tilde{z}^{n} \right\}_{n\in\INo}$, $z\in \IC$ is the standard orthonormal
basis of the normalized $z$-monomials with respect to the unit disc $\mathbb{B}^{+}_{2}$. The
complex derivative $\partial_{z}$ of the Fourier representation of $f$ yields
\begin{equation*}
\partial_z f\,=\,\sum_{n=0}^{\infty}\sqrt{(n+1)(n+2)}\,\tilde{z}^{n}\boldsymbol{\beta}_{n+1},
\end{equation*}
which is up to a real factor and apart from the dimension of the polynomial basis similar to the
representation (\ref{Equation::FourierSeries-Derivative}).

%%*******************************************************************************
\subsubsection{Primitives of $\IH$-holomorphic functions}
%%*******************************************************************************

Let us now focus on the primitivation of $\IH$-holomorphic functions. As briefly outlined in the
beginning, the concept of primitivation by line integrals can in principle not be generalized to
$\IH$ whereby other approaches have to be discussed. In \cite{Bock2009,Bock2009a} a primitivation
operator $\mathcal{P}_{\sIH}$ (see also Theorem \ref{Theorem::CONS_IH-Properties}) was introduced
which is characterized as a formal inversion of the hypercomplex derivative. Actually, this
primitivation operator defines a monogenic primitive for each basis element in the considered
subspace $\mathcal{M}^{+}_{n}(\IH)$ and thus has to be extended by continuity to the whole space
$L_{2}(\IBp;\IH) \cap \ker \bar{\partial}$. It is important to mention that already in
\cite{Cacao2007} a similar primitivation operator on a complete orthonormal systems in
$L_{2}(\IBp;\IH) \cap \ker \bar{\partial}$ (q.v., \cite{Cacao2004a}) was defined that could also be
extended to the whole space. The primitivation operator defined on the orthonormal basis
(\ref{Equation::CONS_IH}) improves the results in \cite{Cacao2007} by certain structural and
orthogonality properties which will be summarized in the following.

\begin{proposition}[\cite{Bock2009,Bock2009a}]
Let $f\in L_{2}(\IBp;\IH) \cap \ker \bar{\partial}$ and consider its associated Fourier series
(\ref{Equation::FourierSeries}). The formal application of the $\mathcal{P}_{\sIH}$-operator, given
by paragraph (\textsc{ii}) of Theorem \ref{Theorem::CONS_IH-Properties}, to each summand of the
series yields the orthogonal series
\begin{equation}\label{Equation::FourierSeries-Primitive}
\mathcal{P}_{\sIH}\,f\,=\,
\sum_{n=0}^{\infty}\sum_{l=0}^{n}\,\sqrt{\frac{2n+3}{(2n+5)(n-l+1)(n+l+2)}}\,\varphi_{n+1,\sIH}^{l}\,
\boldsymbol{\alpha}_{n,l}.
\end{equation}
\end{proposition}
The monogenic primitive $\mathcal{P}_{\sIH}\,f$ of a function $f\in L_{2}(\IBp;\IH)\cap
\ker\bar{\partial}$, defined by the orthogonal series (\ref{Equation::FourierSeries-Primitive}), is
characterized by the next theorem.
\begin{theorem}[\cite{Bock2009,Bock2009a}]
For the primitivation operator $\mathcal{P}_{\sIH}$ and the corresponding monogenic primitive
$\mathcal{P}_{\sIH} f$ of a function $f\in L^{2}(\IBp;\IH)\cap\ker\bar{\partial}$ hold:
\begin{itemize}
\item[(\textsc{i})] The linear operator
\begin{equation*}
\mathcal{P}_{\sIH} : L^{2}(\IBp;\IH) \cap \ker \bar{\partial}\longrightarrow
L^{2}(\IBp;\IH)\cap\ker\bar{\partial}
\end{equation*}
is bounded and its norm is $\|\mathcal{P}_{\sIH}\|\,=\, \sqrt{\frac{3}{10}}$.
\item[(\textsc{ii})] The orthogonal series $\mathcal{P}_{\sIH}f$ converges in $L^{2}(\IBp;\IH)
\cap \ker \bar{\partial}$.
\item[(\textsc{iii})] The monogenic primitive $\mathcal{P}_{\sIH}
f$ is orthogonal to the subset of monogenic constants in $L^{2}(\IBp;\IH) \cap \ker
\bar{\partial}$. Thus
\begin{equation*}
\mathcal{P}_{\sIH}\,f\,\bot\,\left(\,\ker\partial_{0}\cap\ker\bar{\partial}\,\right)\subset\left(\,
L^{2}(\IBp;\IH) \cap \ker \bar{\partial}\,\right).
\end{equation*}
\end{itemize}
\end{theorem}

Once again, let us expose the analogy to the complex theory. On the (orthogonal) basis $\left\{
z^{n} \right\}_{n\in\INo}$ in $L_{2}(\mathbb{B}^{+}_{2};\IC) \cap \ker \bar{\partial}_{z}$ we
define a holomorphic primitivation operator $\mathcal{P}_{z}$, such that for each basis element the
conditions
\begin{equation*}
\mathcal{P}_{z} z^{n}\,=\,\frac{z^{n+1}}{n+1} \quad \text{and} \quad
\partial_{z}\left[ \mathcal{P}_{z} z^{n} \right] = z^{n}
\end{equation*}
hold. Representing a square-integrable holomorphic function $f$ by its Fourier representation with
respect to the normalized basis $\left\{\tilde{z}^{n} \right\}_{n\in\INo}$ we extend the
primitivation operator to the whole space $L_{2}(\mathbb{B}^{+}_{2};\IC)\cap \ker
\bar{\partial}_{z}$ by
\begin{equation*}
\mathcal{P}_{z}
f\,=\,\sum_{n=0}^{\infty}\sqrt{\frac{1}{(n+1)(n+2)}}\,\tilde{z}^{n+1}\boldsymbol{\beta}_{n}.
\end{equation*}
Furthermore, one proves easily that the linear operator
\begin{equation*}
\mathcal{P}_{z} : L_{2}(\mathbb{B}^{+}_{2};\IC) \cap \ker \bar{\partial}_{z}\longrightarrow
L_{2}(\mathbb{B}^{+}_{2};\IC) \cap \ker \bar{\partial}_{z}
\end{equation*}
is bounded and its norm is $\|\mathcal{P}_{z}\|\,=\,\frac{\sqrt{2}}{2}$. It is also easy to see
that the holomorphic primitive $\mathcal{P}_{z} f$ is orthogonal to the subset of holomorphic
constants $L_{2}(\mathbb{B}^{+}_{2};\IC) \cap \ker \bar{\partial}_{z}$ which solely consists of a
complex constant $\boldsymbol{\beta}_{0} = \beta_{0}^{1} + i\beta_{0}^{2}$, where
$\beta_{0}^{1},\beta_{0}^{2}\in\IR$.
%%%%%%%%%%%%%%%%%%%%%%%%%%%%%%%%%%%%%%%%%%%%%%%%%%%%%

%%%%%%%%%%%%%%%%%%%%%%%%%%%%%%%%%%%%%%%%%%%%%%%%%%%%%%%%%%%%%%%%%%%%%%%%%%%%%%%%%%%%%%%%%%%%%
\subsection{An orthogonal Appell basis of inner solid spherical monogenics}
%%%%%%%%%%%%%%%%%%%%%%%%%%%%%%%%%%%%%%%%%%%%%%%%%%%%%%%%%%%%%%%%%%%%%%%%%%%%%%%%%%%%%%%%%%%%%%

Remember now the structural scheme (see Figure \ref{Figure::OverviewSeriesC}) of the canonical
power series expansions in $\IC$ sketched in the beginning of this paper. As it was shown in the
preceding sections, the first branch regarding to the Fourier series expansions could be entirely
generalized to $\IH$ by preserving all the structural properties of the series expansion under
consideration. Let us now focus on the second branch regarding to the Taylor series expansions and
ask whether a series expansion can be constructed using the same basis (\ref{Equation::CONS_IH}).
One essential observation to answer this question is given by the fact that the $(n-l)$-fold
application of the hypercomplex derivative $\partial_{0}$ to an arbitrary basis function
$\varphi_{n,\sIH}^{l}$ always yields a monogenic constant (see, e.g., Figure
\ref{Figure::CONS_IH}). Note that this is similar to the complex case where the $n$-fold complex
derivative of the basis element $z^{n}$ results in a complex constant. As a consequence, one could
define (see \cite{Bock2009,Bock2009a}) another operator
\begin{equation*}
\bar{\partial}_{\s\IC} = \frac{1}{2}\left(\frac{\partial}{\partial x_{1}} +
\mathbf{e}_{3}\frac{\partial}{\partial x_{2}}\right)
\end{equation*}
which is exclusively acting on the set of monogenic constants and mapping a monogenic constant of
degree $n$ to a monogenic constant of degree $(n-1)$. Let us now look again at the complex Taylor
series. As one can easily determine leads the derivation of an arbitrary ansatz function
$\partial_{z} z^{n} = nz^{n-1}$. This special normalization of the ansatz functions is called
Appell property that was generalized already in 1880 by Appell \cite{Appell1880} to more general
polynomial systems, nowadays called Appell systems. He defined a system $\{P_n(x)\}_{n\in\mathbb
N}$ with the property that $\frac{d}{dx}P_n(x)=nP_{n-1}(x)$, $n=1,2,...$\,. This was later on
generalized by many authors in an extensive way. For the case of monogenic polynomials, this idea
was realized at first by Malonek et al. in the papers \cite{Falcao2006, Mal_Fa_2007,
Malonek_Falcao_2007, CacaoMalonek_2008}. The authors defined a special system of homogeneous
monogenic polynomials $P_k(x),\,x\in \mathbb R^n$ with the property that $\frac12\partial
P_n(x)=nP_{n-1}(x)$ and applied this idea to the definition of several elementary functions and the
calculation of combinatorial identities. Claiming now that for each application of the differential
operators $\partial_{0}$ and $\bar{\partial}_{\s\IC}$, respectively, the Appell property must hold
yields the following theorem:

\begin{theorem}[\cite{Bock2009,Bock2009a}]\label{Theorem::AppellSet}
The system of homogeneous monogenic polynomials $\bigl\{A_{n}^{l}:\,l=0,\ldots,n\bigr\}_{n\in\INo}
$, explicitly given by
\begin{equation}\label{Equation::AppellSet}
\left.
\begin{array}{lcl}
A_{n}^{0} & = & \D\frac{2}{n+1} X_{n,0}^{0,\dagger},\\[1ex]
A_{n}^{l} & = & \D\frac{ 2^{l+1}\,n!}{(n+l+1)!}\left( X_{n,0}^{l,\dagger} -
Y_{n,3}^{l,\dagger}\right),\;l=1,\ldots,n,
\end{array}\right\}
\end{equation}
is a complete orthogonal Appell set in $L_{2}(\IBp;\IH) \cap \ker \bar{\partial}$ such that for
each $n\in\IN$
\begin{equation*}
\partial_{0} A_{n}^{l} =\left\{
\begin{array}{ccl}
n\,A_{n-1}^{l} & : & l=0,\ldots,n-1\\
0 & : & l=n
\end{array} \right.
\end{equation*}
and
\begin{equation*}
\bar{\partial}_{\sIC} A_{n}^{n}  = n\,A_{n-1}^{n-1}.
\end{equation*}
\end{theorem}
Figure \ref{Figure::AppellSet} illustrates the action of the differential operators on the
orthogonal basis.
\begin{figure}[htb]
\begin{center}
\includegraphics[scale=1.4,angle=0]{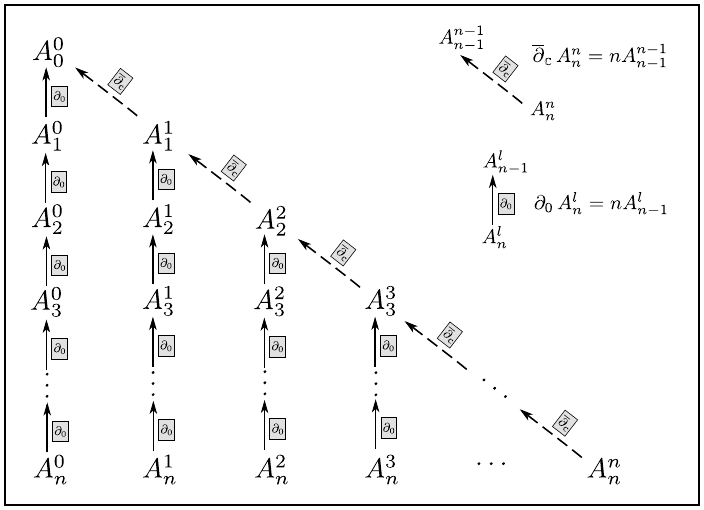}
\captionof{figure}{Action of the differential operators $\partial_{0}$ and $\bar{\partial}_{\sIC}$
on the Appell basis (\ref{Equation::AppellSet}).} \label{Figure::AppellSet}
\end{center}
\end{figure}
Precisely, the application of the hypercomplex derivative $\partial_{0}$ to an arbitrary Appell
polynomial $A_{n}^{l}$ causes a shifting of the degree in a fixed column $l$ whereas the
application of the operator $\bar{\partial}_{\sIC}$ causes a shifting of the degree as well as a
shifting of the column. Again, it should be emphasized that the action of the differential operator
$\bar{\partial}_{\sIC}$ is restricted to the set of monogenic constants and thus, referring to
Figure \ref{Figure::AppellSet}, is mapping along the upper diagonal. As a consequence of the afore
said, one can conclude that for an arbitrary Appell polynomial $A_{n}^{l}$, $l=0,\ldots,n$,
$n\in\INo$ of the system (\ref{Equation::AppellSet}) first the $(n-l)$-fold application of
$\partial_{0}$ and afterwards the $l$-fold application of $\bar{\partial}_{\sIC}$ yields
\begin{equation*}
\bar{\partial}_{\sIC}^{\,l}\,\partial_{0}^{n-l}\,A_{n}^{l} = n!.
\end{equation*}
This property essentially enables the definition of a new Taylor-type series expansion in terms of
the Appell set (\ref{Equation::AppellSet}) at first introduced in \cite{Bock2009,Bock2009a}.
\begin{definition}[Taylor-type series in $L_{2}(\IBp;\IH) \cap \ker \bar{\partial}$]
Let $f\in L_{2}(\IBp;\IH) \cap \ker \bar{\partial}$. The series representation
\begin{equation}\label{Equation::TaylorSeries}
f := \sum_{n=0}^{\infty}\sum_{l=0}^{n} A_{n}^{l} \mathbf{t}_{n,l},\quad\text{with}\quad
%%%%%%%%%%%%%%%%%%%%%%%%
\mathbf{t}_{n,l} =
\frac{1}{n!}\,\bar{\partial}_{\sIC}^{\,l}\,\partial_{0}^{n-l}\,f(\mathbf{x})\,\Bigl|_{\mathbf{x}=\mathbf{0}}
\end{equation}
is called generalized Taylor-type series in $L_{2}(\IBp;\IH) \cap \ker \bar{\partial}$. The
operators $\partial_{0}^{0}$ and $\bar{\partial}_{\sIC}^{\,0}$ are identified with the
corresponding identity operators.
\end{definition}

Finally, let us draw some interesting consequences resulting from this novel Taylor-type series
(\ref{Equation::TaylorSeries}). Firstly, the Taylor series is also an orthogonal series expansions.
More precisely, all summands of the Taylor series form an orthogonal basis in
$L_{2}(\IBp;\IH)\cap\ker\bar{\partial}$, whereby the Taylor coefficients can be directly linked
with the coefficients of a associated Fourier series as for example in the complex case. Due to the
fact that each Appell polynomial is up to a real normalizing factor equivalent to a single
orthonormal basis polynomial (compare, e.g., equation (\ref{Equation::CONS_IH}) and equation
(\ref{Equation::AppellSet})\,) each Fourier coefficient (\ref{Equation::FourierSeries}) of a
function $f\in L_{2}(\IBp;\IH)\cap\ker\bar{\partial}$ can be expressed in terms of the
corresponding Taylor coefficient (\ref{Equation::TaylorSeries}) by the relation
\begin{equation}\label{Equation::Fourier-Taylor}
\boldsymbol{\alpha}_{n,l} \;=\; 2^{l+1}\,\sqrt{\frac{\pi}{(2n+3)\,(n-l)!\,(n+l+1)!}}\;
\bar{\partial}_{\sIC}^{\,l}\,\partial_{0}^{n-l}\,f(\mathbf{x})\,\Bigl|_{\mathbf{x}=\mathbf{0}},
\end{equation}
where $l=0,\ldots,n$ and $n\in\INo$. Secondly, we directly conclude from the equation
(\ref{Equation::Fourier-Taylor}) the transformation of an arbitrary element of the $\IH$-linear
orthonormal system (\ref{Equation::CONS_IH}) into an element of the Appell basis
(\ref{Equation::AppellSet}), that is
\begin{equation}\label{Equation::Transformation_A_Phi}
A_{n}^{l} \;=\; 2^{l+1}\,n!\,\sqrt{\frac{\pi}{(2n+3)\,(n-l)!\,(n+l+1)!}}\;\varphi_{n,\sIH}^{l},\;\;
l=0,\ldots,n,\;n\in\INo.
\end{equation}
For a detailed proof of the relations (\ref{Equation::Fourier-Taylor}) and
(\ref{Equation::Transformation_A_Phi}), see \cite{Bock2009,Bock2009a}.

Obviously, one can show the analogy to the complex theory also for this result. For a function $f
\in L^{2}(\mathbb{B}^{+}_{2};\IC) \cap \ker \bar{\partial}_{z}$, the relations between the complex
Taylor- and Fourier coefficients are given by
\begin{equation*}
\boldsymbol{\beta}_{n} \,=\, \left.\frac{1}{n!}\sqrt{\frac{\pi}{n+1}}\,\partial_{z}^{n}
f(z)\right|_{z=\mathbf{0}}.
\end{equation*}
Restricting ourselves to the special case $l=0$, equality (\ref{Equation::Fourier-Taylor}) reduces
to
\begin{equation*}
\boldsymbol{\alpha}_{n,0} \;=\; \frac{2}{n!}\,\sqrt{\frac{\pi}{(2n+3)\,(n+1)}}\;
\partial_{0}^{n}\,f(\mathbf{x})\,\Bigl|_{\mathbf{x}=\mathbf{0}},
\end{equation*}
which is quite similar to the complex relation.

%%%%%%%%%%%%%%%%%%%%%%%%%%%%%%%%%%%%%%%%%%%%%%%%%%%%%%%%%%%%%%%%%%%%%%%%%%%%%%
%% end of section Complete orthonormal systems of solid spherical monogenics
%%%%%%%%%%%%%%%%%%%%%%%%%%%%%%%%%%%%%%%%%%%%%%%%%%%%%%%%%%%%%%%%%%%%%%%%%%%%%%

%%*******************************************************************************
\section{Recurrence formulae}
%%*******************************************************************************

In view of the practical applicability and an efficient implementation of the orthogonal bases
introduced in the previous section we are also interested in recurrence relations between the
elements of the bases. Such recurrence relations become also very advantageous to prove structural
properties of orthonormal bases of outer spherical monogenics studied later on in this article
(q.v. Section 6). Thus, preparing the upcoming calculations we first provide an explicit
representation of the Appell basis (\ref{Equation::AppellSet}) in spherical coordinates. With
(\ref{Equation::MonogenicPoly_Xm}), we directly obtain for the case $l=0$ the relation
\begin{equation}\label{Equation::RecurrenceAppell_Spher_An0}
r^{n}\,A_{n}^{0}(\boldsymbol\omega)\;=\;\frac{2\,r^{n}}{n+1}\,\left[\A^{0,{n}}+\B^{0,{n}}\cos{\varphi}\,\textbf{e}_1+\B^{0,{n}}\sin{\varphi}\,\textbf{e}_2\right].
\end{equation}
Consider now the basis elements $A_{n}^{l}$, $l=1,\ldots,n$ and determine their spherical part.
With (\ref{Equation::MonogenicPoly_Xm}) and (\ref{Equation::MonogenicPoly_Ym}) we get
\begin{align*}
X_{n,0}^{l} - Y_{n,3}^{l} =\;&\; \A^{l,{n}}\cos l\varphi + \left[ (\B^{l,{n}} + \C^{l,{n}})(
\cos\varphi\cos l\varphi - \sin\varphi\sin l\varphi ) \right]\textbf{e}_{1}\\
%%%%%%%%%%%%%%%%%%%%%%%%%%%%%%%%%%%%%%%%%%
&\; + \left[ (\B^{l,{n}} + \C^{l,{n}})(\cos\varphi\sin l\varphi + \sin\varphi\cos l\varphi )
\right]\textbf{e}_{2} - \A^{l,{n}}\sin l\varphi\,\textbf{e}_{3}\\
%%%%%%%%%%%%%%%%%%%%%%%%%%%%%%%%%%%%%%%%%%
=\;&\; \A^{l,{n}}\bigl[\cos l\varphi - \sin l\varphi\,\textbf{e}_{3}\bigr] + (\B^{l,{n}} +
\C^{l,{n}})\bigl[( \cos\varphi\cos l\varphi - \sin\varphi\sin
l\varphi)\textbf{e}_{1}\\
%%%%%%%%%%%%%%%%%%%%%%%%%%%%%%%%%%%%%%%%%%%
&\;+ (\cos\varphi\sin l\varphi + \sin\varphi\cos l\varphi )\textbf{e}_{2} \bigr].
\end{align*}
For a fixed $n\in\INo$ and $l=0,\ldots,n$, we finally obtain by substituting the relation
(\ref{Equation::RelationABC_II}) into the equation (\ref{Equation::RecurrenceAppell_Spher_An0}) as
well as in the latter equation the general representation
\begin{align}
r^{n}\,A_{n}^{l}(\boldsymbol\omega) \;=\; & \frac{2^{l+1}\,n!\,r^{n}}{(n+l+2)!}\Bigl[ (n+l+2)
\A^{l,{n}}\bigl(\cos l\varphi - \sin l\varphi\,\textbf{e}_{3}\bigr)\nonumber\\
%%%%%%%%%%%%%%%%%%%%%%%%%%%
&\hspace{2.5cm} + \A^{l+1,{n}}\bigl(( \cos\varphi\cos l\varphi - \sin\varphi\sin l\varphi )\textbf{e}_{1}\label{Equation::RecurrenceAppell_Spher_Anl}\\
&\hspace{4.3cm}  + (\cos\varphi\sin l\varphi + \sin\varphi\cos l\varphi )\textbf{e}_{2}\nonumber
\bigr)\Bigr],
\end{align}
where the functions $\A^{l,n}(\theta)$ are explicitly given by (\ref{Equation::HomogMonogPoly_Am}).
The main result of this section is subject of the following theorem.\vspace{0.3cm}
\begin{theorem}
For each $n\in\INo$ the elements of the monogenic Appell basis (\ref{Equation::AppellSet}) satisfy
the recurrence formulae
\begin{align}
\mathbf{x}\,A_{n}^{l} & \;=\; \frac{1}{2(n+1)}\left[ (2n+3)\,A_{n+1}^{l} -
(2l+1)\,\widehat{A_{n+1}^{l}} \right],\label{Equation::RecurrenceAppell_Formel_I}\\[1ex]
%%%%%%%%%%%%%%%%%%%%%%%%%%%%%%%%%%%%%%%%%%%%%%%%%
A_{n+1}^{l} & \;=\; \frac{n+1}{2(n-l+1)(n+l+2)}\left[ (2n+3)\,\mathbf{x}\,A_{n}^{l} +
(2l+1)\,\overline{\mathbf{x}}\,\widehat{A_{n}^{l}}
\right],\label{Equation::RecurrenceAppell_Formel_II}
\end{align}
with $l=0,\ldots,n$. The corresponding anti-monogenic function $\widehat{A_{n}^{l}}$ to the Appell
polynomial $A_{n}^{l}$ is defined by Corollary
\ref{Corollary::Anti-MonogenicFunction-Construction}.
\end{theorem}

\begin{proof} \textbf{The complete proof can be found in \cite{Bock2010a}.} \end{proof}

Finally, let us draw some further conclusions of the constructed one-step recurrence formulae
(\ref{Equation::RecurrenceAppell_Formel_I}) and (\ref{Equation::RecurrenceAppell_Formel_II}).
Firstly, the recurrence formulae relate Appell polynomials of different degree $n$ however the
index $l$ is fixed. Referring to Figure \ref{Figure::AppellSet}, this structurally means that the
elements of the $(l+1)$-th column are recursively generated by the initial elements $A_{l}^{l}$
which are in fact belonging to the subset of monogenic constants. In \cite{Bock2009,Bock2009a} it
was further proved that these constants possess the simple (cartesian) representation
$A_{l}^{l}\;=\; (x_{1} - x_{2}\mathbf{e}_3)^{l}$, $l\in\INo$ and thus are isomorphic to the
anti-holomorphic $\bar{z}$-powers of the complex one-dimensional case. Secondly, a straightforward
computation shows that a combination of the recurrence formulae
(\ref{Equation::RecurrenceAppell_Formel_I}) and (\ref{Equation::RecurrenceAppell_Formel_II}) yields
a two-step recurrence formula using solely $\IH$-holomorphic basis elements.
\begin{corollary}
For each $n\in\IN$ and $l=0,\ldots,n$ the elements of the monogenic Appell basis
(\ref{Equation::AppellSet}) satisfy the two-step recurrence formula
\begin{equation}\label{Equation::RecurrenceAppell_Formel_III}
A_{n+1}^{l} = \frac{n+1}{2(n-l+1)(n+l+2)}\left[ \Bigl((2n+3)\mathbf{x} +
(2n+1)\overline{\mathbf{x}}\Bigr)A_{n}^{l} - 2n\,\mathbf{x}\overline{\mathbf{x}}\,A_{n-1}^{l}
\right]
\end{equation}
with
\begin{equation*}
A_{l+1}^{l} \;=\; \frac{1}{4}\bigl[ (2l+3)\mathbf{x} +
(2l+1)\overline{\mathbf{x}}\bigr]\,A_{l}^{l}\quad\text{and}\quad
%%%%%%%%%%%%%%%%%%%%%%%%%%%%%%%%%
A_{l}^{l} \;=\; (x_{1} - x_{2}\mathbf{e}_{3})^{l}\,.
\end{equation*}
\end{corollary}
Thirdly, the recurrence formulae for the elements of the orthonormal basis
(\ref{Equation::CONS_IH}) are directly obtained by applying the transformation
(\ref{Equation::Transformation_A_Phi}) to formulae (\ref{Equation::RecurrenceAppell_Formel_I}),
(\ref{Equation::RecurrenceAppell_Formel_II}) and (\ref{Equation::RecurrenceAppell_Formel_III}),
respectively. Accordingly, the orthogonal bases (\ref{Equation::CONS_IH}) and
(\ref{Equation::AppellSet}) can now be generated in an very efficient and direct way whereby the
extensive construction process based on the set of inner solid spherical monogenics
(\ref{Equation::SolidSpherMonogenics}) becomes redundant.
%%%%%%%%%%%%%%%%%%%%%%%%%%%%%%%%%%%%%%%%%%%%%%

%%*******************************************************************************
\section{Closed-form representations}
%%*******************************************************************************

In this section the preceding results are applied to construct a closed-form representation for
each element of the orthogonal bases. This idea is motivated by the observation that, for a fixed
index $l\in\INo$, each basis polynomial $A_{n}^{l}$, $n\geq l$ of the $(l+1)$-th column (see Figure
\ref{Figure::AppellSet}) is related to the monogenic constant $A_{l}^{l}$. Hence, it is natural to
ask whether each element of the Appell basis (\ref{Equation::AppellSet}) can be factorized into a
monogenic constant and a polynomial $\IH$-valued rest term.

\begin{theorem}
For each $n\in\IN$ and $l=0,\ldots,n$ the elements of the monogenic Appell basis
(\ref{Equation::AppellSet}) possess the factorization
\begin{equation}\label{Equation:ClosedForm}
A_{n}^{l}=\frac{n!\,l!}{2^{2(n-l)} (n+l+1)!(2l)!}\left[
\sum_{h=0}^{n-l}\frac{(2n-2h+1)!(2l+2h)!}{h!(n-l-h)!(n-h)!(l+h)!}\overline{\mathbf{x}}^{h}\mathbf{x}^{n-l-h}
\right]A_{l}^{l}
\end{equation}
with $A_{l}^{l} \;=\; (x_{1} - x_{2}\mathbf{e}_{3})^{l}$.
\end{theorem}
\begin{proof}
\textbf{The complete proof can be found in \cite{Bock2010a}.}
\end{proof}

Note that already in \cite{Mal_Fa_2007} a factorization of special monogenic polynomials in terms
of $\mathbf{x}$ and $\bar{\mathbf{x}}$ powers was obtained which coincides with the subset of
(\ref{Equation:ClosedForm}) for the special case $l=0$ and thus $A_{0}^{0}=1$. These polynomials,
as it was proved in \cite{NGuerlebeck2008}, are the Fueter-Sce extensions of the complex monomials
$z^{n}$.

%%%%%%%%%%%%%%%%%%%%%%%%%%%%%%%%%%%%%%%%%%%%%%%%%%%%%%%%

%%***********************************************************************
%%***********************************************************************
\section{An $\IH$-linear orthonormal basis of outer solid spherical monogenics}
%%***********************************************************************
%%***********************************************************************

For the generation of complete systems of outer spherical functions the subsequent transformation
is applied.
\begin{definition}[Kelvin transformation in $\IH$]
Let $P_{n}\in\mathcal{M}_{n}^{+}$. The bijective mapping $\mathcal{K} : \mathcal{M}_{n}^{+}
\longrightarrow \mathcal{M}_{n}^{-}$, given by
\begin{equation}\label{V::Formel:HomogMonogFunktionen_KelvinTrans}
Q_{n}(\mathbf{x}) = \mathcal{K}\left( P_{n} \right) =
\frac{\overline{\mathbf{x}}}{|\mathbf{x}|^{3}}\,P_{n}\left(\frac{\overline{\mathbf{x}}}{|\mathbf{x}|^{2}}\right),
\end{equation}
is called \textit{Kelvin transformation in $\IH$}.
\end{definition}
Using this approach is mainly motivated by the fact that for a fixed $n\in\INo$ the complex Kelvin
transformation of the inner polynomial $z^n$ directly yields the corresponding outer function
$z^{-(n+1)}$. For this reason it is natural to ask whether the application of the hypercomplex
Kelvin transformation to the constructed bases (\ref{Equation::CONS_IH}) and
(\ref{Equation::AppellSet}) lead again to complete systems in $\mathbb{B}_{3}^{-}$ and which
properties of the systems are preserved in this process.

First we study the application of the Kelvin transformation
(\ref{V::Formel:HomogMonogFunktionen_KelvinTrans}) to the elements of the orthonormal basis
(\ref{Equation::CONS_IH}).
\begin{theorem}\label{V::Theorem:HomogMonogFunktionen_CONS_IH}
Let $\left\{\varphi_{n,\sIH}^{l} : l = 0,\ldots,n\right\}_{n\in\INo}$ be the $\IH$-linear
orthonormal basis (\ref{Equation::CONS_IH}). For each $n\in \INo$ the homogeneous $\IH$-holomorphic
functions
\begin{equation}\label{V::Formel:HomogMonogFunktionen_CONS_IH}
\varphi_{-(n+2),\sIH}^{l} :=
\sqrt{\frac{2n+1}{2n+3}}\,\mathcal{K}\left(\varphi_{n,\sIH}^{l}\right),\quad l=0,\ldots,n
\end{equation}
with degree of homogeneity $-(n+2)$ are orthonormal in $L^{2}(\mathbb{B}_{3}^{-};\IH) \cap \ker
\bar{\partial}$.
\end{theorem}
\begin{proof} \textbf{The complete proof can be found in \cite{Bock2010b}.} \end{proof}

In \cite{Sudbery1979} it was proved that the subspace $\mathcal{M}_{n}^{-}(\IH)$ of outer spherical
functions with degree of homogeneity $-(n+2)$ has dimension $n+1$. Since the Kelvin transformation
maps the space $L^{2}(\mathbb{B}^{+}_{3};\IH)\cap\ker\bar{\partial}$ onto the space
$L^{2}(\mathbb{B}^{-}_{3};\IH)\cap\ker\bar{\partial}$ we directly conclude from the orthonormality
of the transformed system (\ref{V::Formel:HomogMonogFunktionen_CONS_IH}) the completeness in the
whole space.
\begin{corollary}
The system of homogeneous $\IH$-holomorphic functions $\bigl\{
\varphi_{-(n+2),\sIH}^{l}\,:\,l=0,\ldots,n\bigr\}_{n\in\INo}$ is an orthonormal basis in
$L^{2}(\mathbb{B}_{3}^{-};\IH) \cap \ker \bar{\partial}$.
\end{corollary}
Figure \ref{V::Abbildung:compCONS_IH} illustrates the obtained results for the outer orthonormal
basis (\ref{V::Formel:HomogMonogFunktionen_CONS_IH})  up to now and its connection to the
orthonormal basis (\ref{Equation::CONS_IH}) in
$L^{2}(\mathbb{B}^{+}_{3};\IH)\cap\ker\bar{\partial}$.
\begin{figure}[htb]
\begin{center}
\includegraphics[scale=1.2,angle=0]{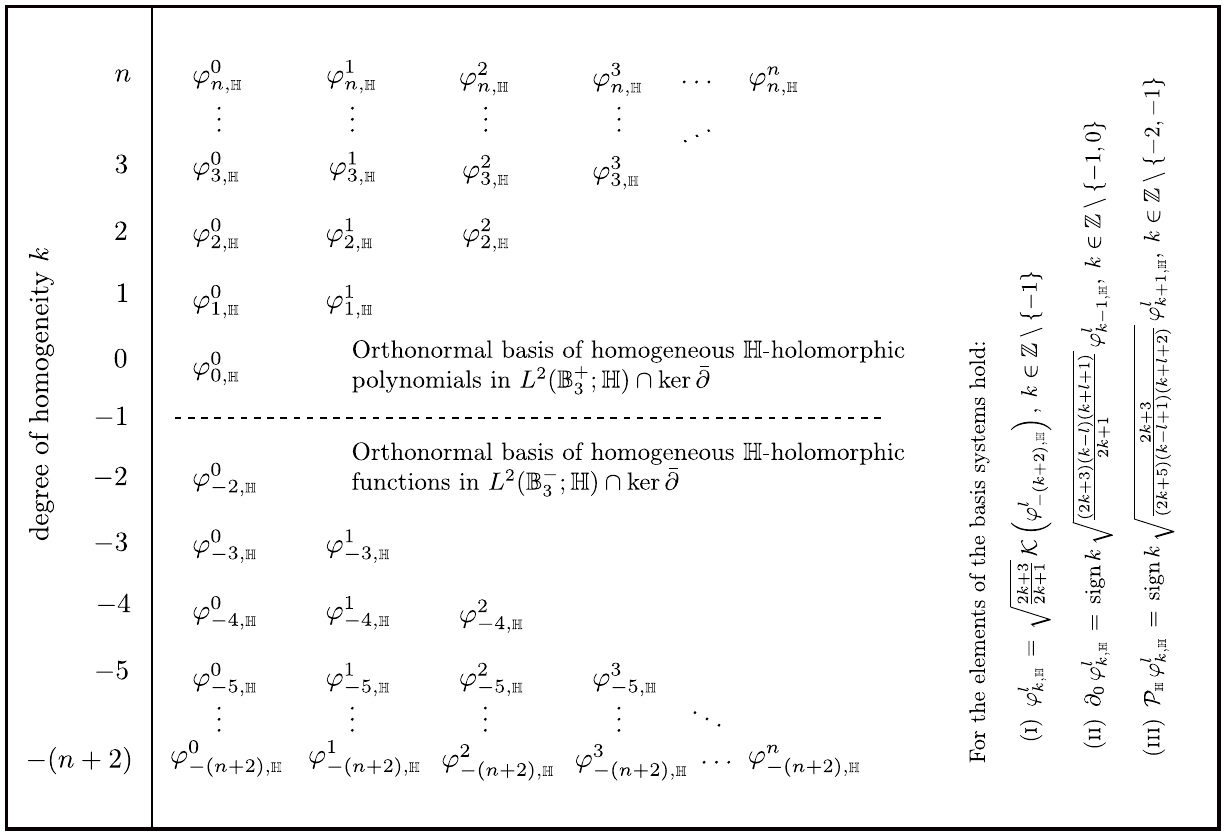}
\caption{Orthonormal bases of inner and outer homogeneous $\IH$-holomorphic functions.}
\label{V::Abbildung:compCONS_IH}
\end{center}
\end{figure}

Let us now study the basis elements of the outer system in particular with regard to the
hypercomplex derivative and primitive. In $\IC$, as is well-known, the derivative of an outer basis
function with degree of homogeneity $-(n+1)$, $n\in\INo$ yields again a real multiple of a basis
function with degree of homogeneity $-(n+2)$. The next theorem shows that also this structural
property is generalized to $\IH$ by the new system of $\IH$-holomorphic functions constructed here.
\begin{theorem}\label{V::Theorem:HomogMonogFunktionen_CONS_IH_Eigenschaften}
For the outer spherical functions $\varphi_{-(n+2),\sIH}^{l}$, $l=0,\ldots,n$ of the orthonormal
basis (\ref{V::Formel:HomogMonogFunktionen_CONS_IH}) the following properties hold:
\begin{itemize}
\item[(\textsc{i})] The application of the hypercomplex derivative $\partial_{0} =
\frac{1}{2}\partial$ to the basis elements of (\ref{V::Formel:HomogMonogFunktionen_CONS_IH}) yields
\begin{equation*}
\partial_{0}\,\varphi_{-(n+2),\sIH}^{l}\,=\,-\sqrt{\frac{(2n+1)(n-l+1)(n+l+2)}{2n+3}}\,\varphi_{-(n+3),\sIH}^{l},
\end{equation*}
where $l=0,\ldots,n$ and $n\in \INo$.
%%%%%%%%%%%%%%%%%%%%%%%%%%%%%%%%%%%%%%%
\item[(\textsc{ii})] The operator $\mathcal{P}_{\sIH}:\mathcal{M}_{n}^{-}(\IH)
\rightarrow \mathcal{M}_{n-1}^{-}(\IH)$, explicitly given by
\begin{equation*}
\mathcal{P}_{\sIH}\,\varphi_{-(n+2),\sIH}^{l}\,=\,
-\sqrt{\frac{2n+1}{(2n-1)(n-l)(n+l+1)}}\,\varphi_{-(n+1),\sIH}^{l},
\end{equation*}
defines a primitive for each element of the $\IH$-linear orthonormal basis, such that for arbitrary
$n\in \IN$ and $l=0,\ldots,n-1$ the relation
$\partial_{0}\left[\mathcal{P}_{\sIH}\,\varphi_{-(n+2),\sIH}^{l}\right] =
\varphi_{-(n+2),\sIH}^{l}$ holds.
\end{itemize}
\end{theorem}

\begin{proof} \textbf{The complete proof can be found in \cite{Bock2010b}.} \end{proof}

Denoting now the degree of homogeneity of the respective inner and outer functions with the uniform
parameter $k$ some of the structural properties (see, i.e., Theorems
\ref{Theorem::CONS_IH-Properties}, \ref{V::Theorem:HomogMonogFunktionen_CONS_IH} and
\ref{V::Theorem:HomogMonogFunktionen_CONS_IH_Eigenschaften}) can be generally stated for both
orthonormal systems (\ref{Equation::CONS_IH}) and (\ref{V::Formel:HomogMonogFunktionen_CONS_IH}).
\begin{corollary}\label{V::Korollar:HomogMonogONS_Eigenschaften}
For the homogeneous $\IH$-holomorphic functions of the complete orthonormal systems
(\ref{Equation::CONS_IH}) and (\ref{V::Formel:HomogMonogFunktionen_CONS_IH}) the following
relations hold:
\begin{itemize}
\item[(\textsc{i})] For arbitrary $k\in\mathbb{Z}\setminus\left\{-1\right\}$, we have the
transformation
\begin{equation*}
\varphi_{k,\sIH}^{l}\,=\,\sqrt{\frac{2k+3}{2k+1}}\,
\mathcal{K}\left(\varphi_{-(k+2),\sIH}^{l}\right),\;l=0,\ldots,|k+1|-1.
\end{equation*}
%%%%%%%%%%%%%%%%%%%%%%%%%%%%%%%%%%%%%%%
\item[(\textsc{ii})] The application of the hypercomplex derivative $\partial_{0} =
\frac{1}{2}\partial$ to an arbitrary element of the systems (\ref{Equation::CONS_IH}) and
(\ref{V::Formel:HomogMonogFunktionen_CONS_IH}) yields
\begin{equation*}
\partial_{0}\,\varphi_{k,\sIH}^{l}\,=\,\sign
k\,\sqrt{\frac{(2k+3)(k-l)(k+l+1)}{2k+1}}\,\varphi_{k-1,\sIH}^{l},
\end{equation*}
where $k\in\mathbb{Z}\setminus\left\{-1,0\right\}$ and
$$l=\left\{\begin{array}{ccl}
0,\ldots,k-1 & : & \sign k = 1,\\
0,\ldots,|k+1|-1 & : & \sign k = -1.
\end{array}
\right.$$
%%%%%%%%%%%%%%%%%%%%%%%%%%%%%%%%%%%%%%%
\item[(\textsc{iii})] The operator $\mathcal{P}_{\sIH}:\mathcal{M}_{n}^{\pm}(\IH)
\rightarrow \mathcal{M}_{n{\pm}1}^{\pm}(\IH)$, explicitly given by
\begin{equation*}
\mathcal{P}_{\sIH}\,\varphi_{k,\sIH}^{l}\,=\,\sign
k\,\sqrt{\frac{2k+3}{(2k+5)(k-l+1)(k+l+2)}}\,\varphi_{k+1,\sIH}^{l},
\end{equation*}
defines a primitive for each element of the $\IH$-linear orthonormal bases
(\ref{Equation::CONS_IH}) and (\ref{V::Formel:HomogMonogFunktionen_CONS_IH}), such that for
arbitrary $k\in\mathbb{Z}\setminus\left\{-2,-1\right\}$ and
$$l=\left\{\begin{array}{ccl}
0,\ldots,k       & : & \sign k = 1,\\
0,\ldots,|k+1|-2 & : & \sign k = -1
\end{array}
\right.$$ the relation $\partial_{0}\left[\mathcal{P}_{\sIH}\,\varphi_{k,\sIH}^{l}\right] =
\varphi_{k,\sIH}^{l}$ holds.
\end{itemize}
\end{corollary}
Regarding Figure \ref{V::Abbildung:compCONS_IH}, the aforementioned mapping and structural
properties between the elements of the bases (\ref{Equation::CONS_IH}) and
(\ref{V::Formel:HomogMonogFunktionen_CONS_IH}) can be easily retraced. Here, it also turns out that
some of the functions take a special case in the scheme. These are on the one hand the inner
functions $\left\{\varphi_{k,\sIH}^{k}\right\}_{k \geq 0}$ with respect to their hypercomplex
derivative and on the other hand the outer functions $\left\{\varphi_{k,\sIH}^{-(k+2)}\right\}_{k
\leq (-2)}$ with respect to their monogenic primitive.

%%%%%%%%%%%%%%%%%%%%%%%%%%%%%%%%%%%%%%%%%

%%***********************************************************************
%%***********************************************************************
\section{An orthogonal Appell basis of outer $\IH$-holo\-mor\-phic functions}
%%***********************************************************************
%%***********************************************************************

The starting point of the construction of an orthogonal Appell basis in
$L^{2}(\mathbb{B}^{-}_{3};\IH)\cap\ker\bar{\partial}$ is taken by relation
(\ref{Equation::Transformation_A_Phi}) between the elements of the orthonormal basis and the Appell
basis, respectively. Applying on both sides of the equation the Kelvin transformation
(\ref{V::Formel:HomogMonogFunktionen_KelvinTrans}) and taking into account
(\ref{V::Formel:HomogMonogFunktionen_CONS_IH}) yields the auxiliary function
\begin{align}
\mathcal{K}\left(A_{n}^{l}\right) &\;=\;
2^{l+1}\,n!\,\sqrt{\frac{\pi}{(2n+3)\,(n-l)!\,(n+l+1)!}}\;\mathcal{K}\left(\varphi_{n,\sIH}^{l}\right)\nonumber\\
%%%%%%%%%%%%%%%%%%%%%%%%%%%
&\;=\;
2^{l+1}\,n!\,\sqrt{\frac{\pi}{(2n+1)\,(n-l)!\,(n+l+1)!}}\;\varphi_{-(n+2),\sIH}^{l}\;=:\;\mathrm{H}_{-(n+2),\sIH}^{\,l}.
\label{V::Formel:HomogMonogFunktionen_Hilfsfunktion_H}
\end{align}
On the basis of Figure \ref{V::Abbildung:compCONS_IH} it becomes clear that the central operation
in constructing an Appell basis in $L^{2}(\mathbb{B}^{-}_{3};\IH)\cap\ker\bar{\partial}$ is now in
contrast to the inner Appell basis (see, i.e., \cite{Bock2009,Bock2009a}) the derivation and not
the primitivation. Therefore, we need the hypercomplex derivative of the aforementioned auxiliary
function. With paragraph $(\textsc{i})$ of Theorem
\ref{V::Theorem:HomogMonogFunktionen_CONS_IH_Eigenschaften} we directly obtain the relation
\begin{equation*}
\partial_{0}\,\mathrm{H}_{-(n+2),\sIH}^{\,l}\;=\;-\frac{(n-l+1)(n+l+2)}{n+1}\,\mathrm{H}_{-(n+3),\sIH}^{\,l}\,.
\end{equation*}
A next difficulty results from the fact that for an arbitrary $n\in\INo$ not all elements of the
subspace $\mathcal{M}_{n+1}^{-}(\IH)$ can be obtained by derivation of the elements in
$\mathcal{M}_{n}^{-}(\IH)$. Regarding Figure \ref{V::Abbildung:compCONS_IH}, these stand alone
elements $\left\{\varphi_{-(n+2),\sIH}^{n}\right\}_{n\in\INo}$ correspond to the Kelvin
transformations of the monogenic constants. Conversely, these functions can be characterized by the
fact that they don't possess a $\IH$-holomorphic primitive in the set of the outer basis functions.
In the complex case, a similar situation can be found for the outer holomorphic function
$\frac{1}{z}$, for which the complex logarithm defines a primitive. Hence, this function doesn't
possess a direct primitive in the basis of the negative $z$-powers as well.

It is for this reason that we now have to think about an appropriate normalization of these special
outer Appell functions. For this purpose, we recall the explicit expressions of the monogenic
constants of the inner Appell system
\begin{equation*}
A_{n}^{n}\;=\; (x_{1} - x_{2}\mathbf{e}_3)^{n}\;=:\; \boldsymbol{\zeta}^{n},\quad n\in\INo\,.
\end{equation*}
Applying the Kelvin transformation (\ref{V::Formel:HomogMonogFunktionen_KelvinTrans}) to the
monogenic Appell constants, we obtain the simple relations
\begin{equation}\label{V::Formel:HomogMonogFunktionen_Transformation_Ann}
A_{-(n+2)}^{n}\;:=\;\mathcal{K}\left(A_{n}^{n}\right)\;=\;(-1)^{n}\,
\frac{\bar{\mathbf{x}}\,\boldsymbol{\zeta}^{n}}{|\mathbf{x}|^{2n+3}},\quad n\in\INo.
\end{equation}
As one can easily conclude from equation (\ref{V::Formel:HomogMonogFunktionen_Transformation_Ann}),
the transformation of the Appell polynomial $A_{0}^{0}$ yields a real multiple of the Cauchy kernel
\begin{equation}\label{IV::Formel:Cauchy-Kern}
E_{2}(\mathbf{x})\;=\; \frac{1}{4\pi}\frac{\bar{\mathbf{x}}}{|\mathbf{x}|^3},\quad\mathbf{x}\neq 0
\end{equation}
which is analogous to the complex case. The images of the other transformed constants $A_{n}^{n}$,
$n>0$ can then be characterized as products of the Cauchy kernel and the function
$\frac{4\pi\,(-1)^{n}\,\boldsymbol{\zeta}^{n}}{|\mathbf{x}|^{2n}}$. Caused by the fact that these
properties naturally arise by the transformation of the Appell constants we take in the first
instance the relation (\ref{V::Formel:HomogMonogFunktionen_Transformation_Ann}) as a basis for the
normalization of the outer Appell functions. Thus, the construction auf the outer Appell system has
to be done column-wise (q.v., Figure \ref{V::Abbildung:compCONS_IH}). For an arbitrary $p\in\INo$
and with $\gamma_{n,l} = -\frac{(n-l+1)(n+l+2)}{n+1}$, we get for the ($p+1$)-th column:
\begin{equation*}
\begin{array}{rcl}
\Ah_{-(p+2)}^{p} & = &\mathrm{H}_{-(p+2),\sIH}^{p}\\
%%%%%%%%%%%%%%%%%%%%%%%%%%%%%
\Ah_{-(p+3)}^{p} & = &\D\frac{1}{(-(p+2))}\,\partial_{0} \Ah_{-(p+2)}^{p} \;=\; \frac{1}{(-(p+2))}\,\gamma_{p,p}\,\mathrm{H}_{-(p+3),\sIH}^{p}\\
%%%%%%%%%%%%%%%%%%%%%%%%%%%%%
\Ah_{-(p+4)}^{p} & = &\D\frac{1}{(-(p+3))}\,\partial_{0} \Ah_{-(p+3)}^{p} \;=\; \frac{1}{(-(p+2))(-(p+3))}\,\gamma_{p,p}\,\gamma_{p+1,p}\,\mathrm{H}_{-(p+4),\sIH}^{p}\\
%%%%%%%%%%%%%%%%%%%%%%%%%%%%%
\vdots\hspace{0.3cm} & &\hspace{1cm}\vdots \\
%%%%%%%%%%%%%%%%%%%%%%%%%%%%%
\Ah_{-(n+2)}^{l} & = &\D\frac{(l+1)!}{(-1)^{n-l}(n+1)!}
\prod_{j=l}^{n-1}\gamma_{j,l}\,\mathrm{H}_{-(n+2),\sIH}^{l}.
\end{array}
\end{equation*}
In regard of relation (\ref{V::Formel:HomogMonogFunktionen_Hilfsfunktion_H}), we finally obtain
after simplification an outer Appell system
\begin{equation*}
\Ah_{-(n+2)}^{l}\;=\;\frac{l!\,(l+1)!\,(n+l+1)!\,(n-l)!}{n!\,(n+1)!\,(2l+1)!}\,\mathcal{K}\left(A_{n}^{l}\right),\;
l=0,\ldots,n,\;n\in\INo
\end{equation*}
in terms of the Kelvin transformed inner Appell set (\ref{Equation::AppellSet}). Having in mind
that a column-wise multiplication with an arbitrary $\IH$-valued constant with real coordinates
doesn't affect the Appell property, we can neglect those coefficients in the above representation
of the Appell set which only depend on the parameter $l$. Consequently, the outer Appell system
becomes
\begin{equation}\label{V::Formel:HomogMonogFunktionen_AppellMenge_oA}
A_{-(n+2)}^{l}\;=\;\frac{(n+l+1)!\,(n-l)!}{n!\,(n+1)!}\,\mathcal{K}\left(A_{n}^{l}\right),\;
l=0,\ldots,n,\;n\in\INo,
\end{equation}
which is particularly characterized by the following theorem.
\begin{theorem}\label{V::Theorem:HomogMonogFunktionen_AppellMenge}
The system of outer homogeneous $\IH$-holomorphic functions $\bigl\{A_{-(n+2)}^{l}:
l=0,\ldots,n\bigr\}_{n\in\INo}$ is an orthogonal Appell basis in $L^{2}(\mathbb{B}_{3}^{-};\IH)
\cap \ker \bar{\partial}$, such that for each $n\in\INo$
\begin{equation*}
\partial_{0} A_{-(n+2)}^{l} = -(n+2)\,A_{-(n+3)}^{l},\quad l=0,\ldots,n
\end{equation*}
and
\begin{equation*}
A_{-(n+2)}^{n}\;=\;\frac{(-1)^{n}\,(2n+1)!}{n!\,(n+1)!}\,
\frac{\bar{\mathbf{x}}\,\boldsymbol{\zeta}^{n}}{|\mathbf{x}|^{2n+3}},\quad
\boldsymbol{\zeta}\,=\,x_{1}-x_{2}\mathbf{e}_{3}.
\end{equation*}
\end{theorem}
\begin{proof}
The proof follows by construction.
\end{proof}

For a fixed $n\in\INo$, the direct relation between the outer Appell functions
(\ref{V::Formel:HomogMonogFunktionen_AppellMenge_oA}) and the elements of the outer orthonormal
basis (\ref{V::Formel:HomogMonogFunktionen_CONS_IH}) is then equivalently given by
\begin{equation}\label{V::Formel:HomogMonogFunktionen_Umrechnung_A_Phi}
A_{-(n+2)}^{l}\,=\,\frac{2^{l+1}}{(n+1)!}\sqrt{\frac{\pi\,(n-l)!\,(n+l+1)!}{2n+1}}\,
\varphi_{-(n+2),\sIH}^{l},\;l=0,\ldots,n.
\end{equation}
From that and due to relation (\ref{V::Formel:HomogMonogFunktionen_CONS_IH}), we also obtain the
link to the elements of the inner orthonormal basis (\ref{Equation::CONS_IH}).

Finally, let us consider an alternative approach to generate the outer Appell basis. Here, the
elements are constructed by hypercomplex derivation and thus independent from the Kelvin
transformation acting on the inner polynomials. Using the relations for the transformed monogenic
constants given in Theorem \ref{V::Theorem:HomogMonogFunktionen_AppellMenge}, we get
\begin{equation*}
A_{-(n+2)}^{l}\,=\,\frac{(-1)^{n}(2l+1)!}{l!\,(n+1)!}\partial_{0}^{n-l}\left(
\frac{\bar{\mathbf{x}}\,\boldsymbol{\zeta}^{l}}{|\mathbf{x}|^{2l+3}}
\right)\,=\,\frac{(-1)^{n}(2l+1)!}{l!\,(n+1)!}\left[\frac{\partial^{n-l}}{\partial x_{0}^{n-l}}
\frac{\bar{\mathbf{x}}}{|\mathbf{x}|^{2l+3}} \right]\boldsymbol{\zeta}^{l},
\end{equation*}
where $l=0,\ldots,n$ and $n\in\INo$. In this context it must be emphasized that already in
\cite{BDS} partial derivatives of the Cauchy kernel were studied. As one can easily conclude from
the last equation, the partial derivatives $\frac{\partial^{n}}{
\partial x_{0}^{n}}$ of the unnormed Cauchy kernel correspond to the special case $l=0$.

%%%%%%%%%%%%%%%%%%%%%%%%%%%%%%%%

%%***********************************************************************
%%***********************************************************************
\section{Laurent series expansions of $\IH$-holomorphic functions}
%%***********************************************************************
%%***********************************************************************

Let us now consider the construction of a new type of Laurent series expansions of
$\IH$-holomorphic functions using in particular the aforementioned function systems. At first, we
obtain from the proof of the Theorem \ref{V::Theorem:HomogMonogFunktionen_CONS_IH} that the surface
integral of two arbitrary elements of the orthonormal systems (\ref{Equation::CONS_IH}) and
(\ref{V::Formel:HomogMonogFunktionen_CONS_IH}) with respect to the unit sphere $S^{2}$ yields
\begin{equation*}
<\varphi_{k,\sIH}^{l},\varphi_{q,\sIH}^{m}>_{\s (S^{2},\IH)}\;=\;\left\{
\begin{array}{ccl}
|2k+3| & : & k = q \wedge l = m,\\
0 & : & k \neq q \vee l \neq m,
\end{array}\right.
\end{equation*}
where $l=0,\ldots,|k+1|-1$, $m=0,\ldots,|q+1|-1$ and $k,q \in
\mathbb{Z}\setminus\left\{-1\right\}$. Thus, on the basis of the systems (\ref{Equation::CONS_IH})
and (\ref{V::Formel:HomogMonogFunktionen_CONS_IH}) we make a general Fourier ansatz with respect to
$S^{2}$ resulting in the following series expansion
\begin{equation}\label{V::Formel:Laurentreihe_Einheitskugelschale}
f(\mathbf{x}) \;:=\; \sum_{{k=-\infty \atop k\neq
-1}}^{\infty}\sum_{l=0}^{|k+1|-1}\,\varphi_{k,\sIH}^{l}(\mathbf{x})\,\boldsymbol{\gamma}^{\ast}_{k,l},\quad
\boldsymbol{\gamma}^{\ast}_{k,l}\;=\;\frac{1}{|2k+3|}\int_{S^{2}}\overline{\varphi_{k,\sIH}^{l}(\mathbf{y})}\,f(\mathbf{y})\,d\sigma\,.
\end{equation}
This orthogonal series can be viewed as a generalized Laurent series expansion of the
$\IH$-holomorphic function $f$ in the domain of the unit spherical shell $\Omega_{S^{2}} =
\left\{\mathbf{x}\;\big|\;\xi < |\mathbf{x}| < 1 \right\}$. Since the outer radius is fixed by the
unit sphere $S^{2}$, it holds $0\leq\xi < 1$.

In the following, the series expansion (\ref{V::Formel:Laurentreihe_Einheitskugelschale}) will be
generalized to an arbitrary spherical shell and the interplay with the monogenic Taylor expansion
(\ref{Equation::TaylorSeries}) as well as with Cauchy's integral formula
\begin{equation}\label{IV::Formel:Integralformel_von_Cauchy}
\int_{\partial\Omega} E_{2}(\mathbf{y-x})\,d\mathbf{y}^{\ast}\,f(\mathbf{y})\;=\;\left\{
\begin{array}{ccl}
f(\mathbf{x}) & : & \mathbf{x} \in \Omega,\\
0 & : & \mathbf{x} \in \IR^{3}\setminus\overline{\Omega}
\end{array}\right.
\end{equation}
will be shown. At first we replace each basis function in the Laurent coefficients of the series
expansion (\ref{V::Formel:Laurentreihe_Einheitskugelschale}) equivalently by its associated Kelvin
transformed function. Due to paragraph $(\textsc{i})$ of Corollary
\ref{V::Korollar:HomogMonogONS_Eigenschaften} we get
\begin{align*}
\boldsymbol{\gamma}^{\ast}_{k,l} \;=\; &
\frac{1}{|2k+3|}\sqrt{\frac{2k+3}{2k+1}}\int_{S^{2}}\overline{\frac{\bar{\mathbf{y}}}{|\mathbf{y}|^{3}}\varphi_{-(k+2),\sIH}^{l}\left(\frac{\bar{\mathbf{y}}}{|\mathbf{y}|^{2}}\right)}\,f(\mathbf{y})\,d\sigma\\[1ex]
%%%%%%%%%%%%%%%%%%%%%%%%%%
\;=\; &
\frac{1}{\sqrt{(2k+3)(2k+1)}}\int_{S^{2}}\overline{\varphi_{-(k+2),\sIH}^{l}\left(\frac{\bar{\mathbf{y}}}{|\mathbf{y}|^{2}}\right)}\,\frac{\mathbf{y}}{|\mathbf{y}|^{3}}\,f(\mathbf{y})\,d\sigma.
\end{align*}
Let now $S^{2}_{\!\rho} = \left\{\mathbf{x}\;\big|\;|\mathbf{x}| = \rho \right\}$ be the sphere of
radius $\rho$ and $d\mathbf{y}^{\ast} = \D\frac{\mathbf{y}}{|\mathbf{y}|}d\sigma$ be the normalized
surface element. The Laurent coefficients regarding the integration over $S^{2}_{\!\rho}$ become
equivalently
\begin{equation}\label{V::Formel:Laurentreihe_Koeffizient_Sr}
\boldsymbol{\gamma}^{\ast}_{k,l} \;=\;
\frac{1}{\sqrt{(2k+3)(2k+1)}}\int_{S^{2}_{\!\rho}}\overline{\varphi_{-(k+2),\sIH}^{l}(\bar{\mathbf{y}})}\,d\mathbf{y}^{\ast}f(\mathbf{y}).
\end{equation}
From that we consider the orthogonal series (\ref{V::Formel:Laurentreihe_Einheitskugelschale}) in
the form
\begin{equation*}
f(\mathbf{x})\,=\,f^{a}(\mathbf{x})\,+\,f^{b}(\mathbf{x})\,=\,\underbrace{\sum_{m=0}^{\infty}\sum_{l=0}^{m}\,\varphi_{-(m+2),\sIH}^{l}(\mathbf{x})\,\boldsymbol{\gamma}^{\ast}_{-(m+2),l}}
_{\D\text{\textcircled{\raisebox{1pt}{\scriptsize a}}}} \; + \;
\underbrace{\sum_{n=0}^{\infty}\sum_{l=0}^{n}\,\varphi_{n,\sIH}^{l}(\mathbf{x})\,\boldsymbol{\gamma}^{\ast}_{n,l}}
_{\D\text{\textcircled{\raisebox{0pt}{\scriptsize b}}}}
\end{equation*}
and call the partial series \textcircled{\raisebox{1pt}{\scriptsize a}} the \textit{principle part}
and \textcircled{\raisebox{0pt}{\scriptsize b}} the \textit{secondary part} of the Laurent series
expansion of $f$. In consideration of the Laurent coefficients
(\ref{V::Formel:Laurentreihe_Koeffizient_Sr}), with $k=-(m+2)$, $m\in\INo$ and accordingly $k=n$,
$n\in\INo$, we obtain by substitution of the relations (\ref{Equation::Transformation_A_Phi}) and
(\ref{V::Formel:HomogMonogFunktionen_Umrechnung_A_Phi}) for the principle part
\begin{align*}
f^{a}(\mathbf{x})\;=\;&\sum_{m=0}^{\infty}\sum_{l=0}^{m}\,A_{-(m+2)}^{l}(\mathbf{x})\;\frac{m+1}{4^{l+1}\,\pi}\int_{S^{2}_{\!\rho}}\overline{A_{m}^{l}(\bar{\mathbf{y}})}\,d\mathbf{y}^{\ast}f(\mathbf{y})
\intertext{and analogously for the secondary part of the series}
f^{b}(\mathbf{x})\;=\;&\sum_{n=0}^{\infty}\sum_{l=0}^{n}\,A_{n}^{l}(\mathbf{x})\;\frac{n+1}{4^{l+1}\,\pi}\int_{S^{2}_{\!\rho}}\overline{A_{-(n+2)}^{l}(\bar{\mathbf{y}})}\,d\mathbf{y}^{\ast}f(\mathbf{y}).
\end{align*}
Consequently, the main result of this subsection is given by the following theorem:
\begin{theorem}[Laurent series expansion in $\IH$]
Let $f$ be $\IH$-holomorphic in the spherical shell
\begin{equation*}
 \Omega_{\xi} = \left\{\mathbf{x}\;\big|\; \xi_{1} < |\mathbf{x}| < \xi_{2} \quad\text{with}\quad 0\leq \xi_{1} < \xi_{2} \leq\infty
 \right\}\,.
\end{equation*}
Then $f$ can be represented in $\Omega_{\xi}$ as a Laurent series expansion
\begin{equation}\label{V::Formel:Laurentreihe_Kugelschale}
f(\mathbf{x}) \;:=\; \sum_{{k=-\infty \atop k\neq
-1}}^{\infty}\sum_{l=0}^{|k+1|-1}\,A_{k}^{l}(\mathbf{x})\,\boldsymbol{\gamma}_{k,l},\quad
\boldsymbol{\gamma}_{k,l}\;=\;\frac{|k+1|}{4^{l+1}\pi}\int_{S^{2}_{\!\rho}}\overline{A_{-(k+2)}^{l}(\bar{\mathbf{y}})}\,d\mathbf{y}^{\ast}f(\mathbf{y})
\end{equation}
around the origin. Here, $S^{2}_{\!\rho}$ denotes an arbitrary sphere of distance $\rho$ from the
origin lying inside $\Omega_{\xi}$.
\end{theorem}

Finally, let us remark on some interesting properties of the presented series expansion. In analogy
to the complex case it can be shown that the secondary part of the Laurent series expansion
(\ref{V::Formel:Laurentreihe_Kugelschale}) corresponds to the associated Taylor series expansion.
For $n\in\INo$ and $l=0,\ldots,n$, the application of the hypercomplex differential operators
$\bar{\partial}_{\sIC}^{\,l}\,\partial_{0}^{n-l}$ on both sides of Cauchy's integral formula
(\ref{IV::Formel:Integralformel_von_Cauchy}) yields the relation
\begin{equation*}
\bar{\partial}_{\sIC}^{\,l}\,\partial_{0}^{n-l}\,f(\mathbf{x}) \;=\;
\frac{1}{4\pi}\int_{S^{2}_{\!\rho}}\bar{\partial}_{\sIC}^{\,l}\,\partial_{0}^{n-l}\,\left(\frac{\overline{\mathbf{y-x}}}{|\mathbf{y-x}|^{3}}\right)\,d\mathbf{y}^{\ast}f(\mathbf{y}).
\end{equation*}
Further, it can be shown that
\begin{equation*}
\bar{\partial}_{\sIC}^{\,l}\,\partial_{0}^{n-l}\,\left(\frac{\overline{\mathbf{y-x}}}{|\mathbf{y-x}|^{3}}\right)\,=\,
\bar{\partial}_{\sIC}^{\,l}\,\partial_{0}^{n-l}\,A_{-2}^{0}(\mathbf{y-x})\,=\,\frac{(n+1)!}{4^l}\,\overline{A_{-(n+2)}^{l}\bigl(\overline{\mathbf{y-x}}\bigr)}
\end{equation*}
holds. In the context of the Taylor series expansion (\ref{Equation::TaylorSeries}), we obtain, in
a manner of speaking, for the $n$-th derivative of the function $f$ the relation
\begin{equation*}
\bar{\partial}_{\sIC}^{\,l}\,\partial_{0}^{n-l}\,f(\mathbf{x}) \;=\;
\frac{(n+1)!}{4^{l+1}\pi}\int_{S^{2}_{\!\rho}}\overline{A_{-(n+2)}^{l}\bigl(\overline{\mathbf{y-x}}\bigr)}\,d\mathbf{y}^{\ast}f(\mathbf{y})
\end{equation*}
and thus for the Taylor coefficients
\begin{equation*}
\frac{1}{n!}\bar{\partial}_{\sIC}^{\,l}\,\partial_{0}^{n-l}\,f(\mathbf{x})\Bigl|_{\mathbf{x}=\mathbf{0}}
\;=\;
\frac{n+1}{4^{l+1}\pi}\int_{S^{2}_{\!\rho}}\overline{A_{-(n+2)}^{l}\bigl(\bar{\mathbf{y}}\bigr)}\,d\mathbf{y}^{\ast}f(\mathbf{y}),
\end{equation*}
where $l=0,\ldots,n$ and $n\in\INo$.

Once again, we expose the strong analogy of the Appell polynomials $\left\{ A_{k}^{0}
\right\}_{k\in\IZ\setminus\{-1\}}$ to the complex Laurent series expansion. In this connection, we
remind also on the direct relation of this subset with the Appell polynomials constructed by H.R.
Malonek et al. (q.v. \cite{Falcao2006,Malonek_Falcao_2007,Mal_Fa_2007}) and with the Fueter-Sce
extensions of the complex monomials $z^{n}$ (q.v. \cite{NGuerlebeck2008}), respectively, which
gives a natural explanation of this issue. Having further in mind that each of this functions
belongs to the subspace $\mathcal{A}$ and thus, as in the complex case, for each element the
relation $\overline{A_{k}^{0}(\bar{\mathbf{x}})} = A_{k}^{0}(\mathbf{x})$ holds, the Laurent series
expansion (\ref{V::Formel:Laurentreihe_Kugelschale}) for the case $l=0$ finally reduces to
\begin{equation*}
f(\mathbf{x}) \;=\; \sum_{{k=-\infty \atop k\neq
-1}}^{\infty}\,A_{k}^{0}(\mathbf{x})\,\boldsymbol{\gamma}_{k,0},\quad
\boldsymbol{\gamma}_{k,0}\;=\;\frac{|k+1|}{4\pi}\int_{S^{2}_{\!\rho}}A_{-(k+2)}^{0}(\mathbf{y})\,d\mathbf{y}^{\ast}f(\mathbf{y})\,.
\end{equation*}
A comparison with the complex Laurent series expansion for the annulus around the origin
\begin{equation*} f(z) \;=\; \sum_{k=-\infty}^{\infty}z^{k}\,\boldsymbol{\gamma}_{k}, \quad
\boldsymbol{\gamma}_{k}\;=\;\frac{1}{2\pi i}\int_{S^{1}_{\!\rho}}
\frac{f(\zeta)}{\zeta^{k+1}}\,d\zeta
\end{equation*}
clearly shows the similarity of both series expansions.

%%***********************************************************************
\section{Conclusions}
%%***********************************************************************

In this article a constructive approach to generalize the canonical series expansions (Fourier,
Taylor, Laurent) of the complex one-dimensional case to dimension 3 was presented. In the framework
of hypercomplex functions theory, in particular the theory of $\IH$-holomorphic functions, very
recent orthogonal bases of solid spherical monogenics \cite{Bock2009,Bock2009a} were defined which
fully generalize the behavior of the holomorphic $z$-powers with regard to the derivation and
primitivation. Further, these bases allow the definition of a Fourier and a new orthogonal
Taylor-type series expansion with the special property that their hypercomplex derivative and
primitive are again orthogonal series and that the coefficients of both series expansions can be
explicitly (one-to-one) linked with each other (see, Figure \ref{Figure::SeriesExpansionsH}).
\begin{figure}[htb]
\begin{center}
\includegraphics[scale=.5,angle=0]{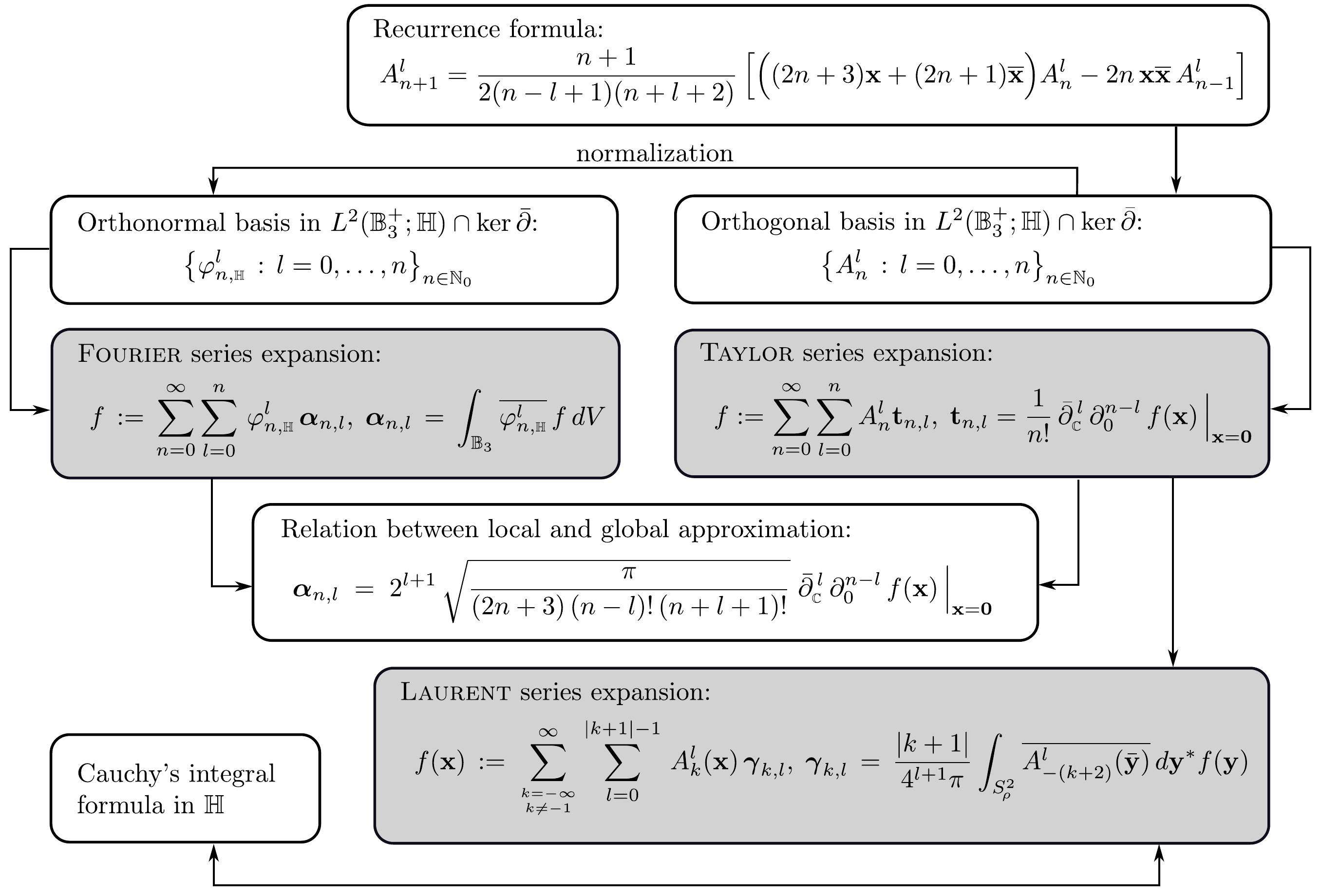}
\caption{ Power and Laurent series expansions in $\IH$. } \label{Figure::SeriesExpansionsH}
\end{center}
\end{figure}
This is a direct consequence of the aforementioned properties and corresponds to the complex power
series expansions as well. Recently, it was shown that the orthonormal basis presented here has
some deeper theoretical meaning in group theory in particular in relation to the so-called
Gelfand-Tsetlin basis. First works in this context can be found in \cite{Bock2010c} and
\cite{Lavicka2010}. In view of the practical applicability of these series expansion one could
prove very compact recurrence formulae (see, e.g., the two-step formula in Figure
\ref{Figure::SeriesExpansionsH}) as well as a closed-form representation for the elements of the
bases. A very interesting fact is that each basis polynomial of degree $n$ and signature $l$ is
recursively generated from the monogenic constant $(x_{1}-x_{2}\mathbf{e}_{3})^{l}$ of degree $l$
in $(n-l)$ recursion steps. In analogy to the complex theory it was then natural to ask for the
spatial analogue to the holomorphic $z$-powers of negativ degree. Here, a new orthonormal basis of
outer solid spherical monogenics in $L^{2}(\mathbb{B}_{3}^{-};\IH) \cap \ker \bar{\partial}$ was
constructed by applying the Kelvin transformation and some real normalization factor to the
elements of the orthonormal basis in $L^{2}(\mathbb{B}_{3}^{+};\IH) \cap \ker \bar{\partial}$.
Noteworthy here is that the transformation was also preserving the properties regarding the
hypercomplex derivative and the primitive. Taking now into consideration that the Laurent series
expansion in $\IC$ is defined in terms of the non-normalized $z$-powers and thus having the Appell
property, one could also construct an outer Appell basis in $L^{2}(\mathbb{B}_{3}^{-};\IH) \cap
\ker \bar{\partial}$ which is consistent with the Appel basis in $L^{2}(\mathbb{B}_{3}^{+};\IH)
\cap \ker \bar{\partial}$. The latter enabled the explicit definition of a new orthogonal Laurent
series expansion in the spherical shell. Finally, the interplay of the Laurent series expansion
with Cauchy's integral formula and the Taylor series expansion was emphasized. Analogously as in
the complex case, it was shown that each Taylor coefficient can be expressed by some surface
integral over a higher order derivative of the Cauchy kernel which furthermore coincides with the
respective Laurent coefficient of the secondary part of the series expansion.

%%*******************************************************************************
\section*{Acknowledgements}
%%*******************************************************************************
The author expresses his gratitude to his mentor Prof. K. G\"{u}rlebeck for the helpful discussions.

%%%%%%%%%%%%%%%%%%%%%%%%%%%%%%%%

%---------------------------------------------------------------------------------------

%---------------------------------------------------------------------------------------
\end{document}